\documentclass[a4paper,11pt]{article}
\usepackage[all]{xy}
\usepackage[T1]{fontenc}
\usepackage{tikz-cd}
\usepackage[utf8]{inputenc}
\usepackage{lmodern}
\usepackage{microtype}
\usepackage{mathtools}
\usepackage{amsfonts}
\usepackage{amsmath}
\usepackage{amssymb}
\usepackage{pictexwd,dcpic}
\usepackage{thmtools}
\newtheorem{theorem}{Theorem}[section]
\newtheorem{lemma}[theorem]{Lemma}
\newtheorem{proposition}[theorem]{Proposition}
\newtheorem{corollary}[theorem]{Corollary}

\newtheorem{conject}[theorem]{Conjecture}

\newenvironment{proof}[1][Proof]{\begin{trivlist}
\item[\hskip \labelsep {\bfseries #1}]}{\end{trivlist}}
\newenvironment{conjecture}[1][Conjecture]{\begin{trivlist}
\item[\hskip \labelsep {\bfseries #1}]}{\end{trivlist}}

\newcommand{\qed}{\nobreak \ifvmode \relax \else
    \ifdim\lastskip<1.5em \hskip-\lastskip
     \hskip1.5em plus0em minus0.5em \fi \nobreak
      \vrule height0.75em width0.5em depth0.25em\fi}
\usepackage[top=2cm,bottom=2cm,left=2.5cm,right=2cm]{geometry}

 \DeclareFontFamily{U}{wncy}{}
    \DeclareFontShape{U}{wncy}{m}{n}{<->wncyr10}{}
    \DeclareSymbolFont{mcy}{U}{wncy}{m}{n}
    \DeclareMathSymbol{\Sh}{\mathord}{mcy}{"58}
\usepackage{hyperref}
\usepackage{fancyhdr}
\date{}
  \title{\textbf{A Coates--Sinnott-type Theorem for First Derivatives of Artin \(L\)-Functions}}
  \author{Saad EL BOUKHARI}
\begin{document}
\maketitle
\date{}
\footnotetext[1]{S. El Boukhari. University Marie et Louis Pasteur, Lab. of Math. of Besançon, Besançon, 25000, France. 
		\textit{E-mails:}  \texttt{saad.el\_boukhari@umlp.fr}, \texttt{saadelboukhari1234@gmail.com}}
\footnotetext[2]{The author declares that he has no conflict of interest to disclose.}
\begin{abstract}
Let \(K/k\) be a finite abelian extension of number fields with Galois group \(G\) and let \(n\geq 2\). We prove, assuming the relevant \(p\)-part of the equivariant Tamagawa number conjecture, first-derivative analogues of the Deligne--Ribet integrality theorem and of the Coates--Sinnott conjecture. We construct a rank-one leading term from the first derivatives at \(s=1-n\) of the \(S\)-truncated Artin \(L\)-functions and show that it satisfies an integral annihilation property. We then attach to this leading term a fractional ideal of \(\mathbb Q_p[G]\) and prove that, up to the
natural torsion factor coming from \(K_{2n-1}(O_K)\), this ideal annihilates the even \(K\)-group \(K_{2n-2}(O_{K,S})\). The proof uses determinant methods, \(\Sigma\)-modified étale complexes, and a cancellation argument which removes
the auxiliary Euler factors.
\end{abstract}
\section{Introduction}
Let \(K/k\) be a finite abelian extension of number fields with Galois group
\(G\), and let \(S\) be a finite set of places of \(k\) containing the
archimedean places and all places ramified in \(K/k\). The aim of this paper is
to prove first-derivative analogues of two classical results on equivariant
special values of Artin \(L\)-functions: the integrality theorem of Deligne and
Ribet and the Coates--Sinnott conjecture.

\medskip

\noindent Special values of Artin \(L\)-functions often give annihilators of
arithmetic groups. This phenomenon is already visible in the Brumer--Stark
conjecture \cite{BrumerStark}: in its classical setting, Stickelberger elements
constructed from values at \(s=0\) are predicted to annihilate suitable ideal
class groups. In this way, Artin \(L\)-values do not only control numerical
invariants such as orders of groups; they also produce explicit annihilation
relations over group rings.

\medskip

\noindent At negative integers, the corresponding arithmetic objects are higher
algebraic \(K\)-groups. For \(n\geq 2\), one considers the equivariant
Stickelberger element (see \S \ref{Sec2.3} below)
\[
\theta_{K/k,S}(1-n)
=
\sum_{\chi\in\widehat G}
L_{K/k,S}(1-n,\chi^{-1})e_\chi
\in \mathbb C[G].
\]
By a theorem of Siegel \cite{Siegel}, one has in fact
\[
\theta_{K/k,S}(1-n)\in\mathbb Q[G].
\]
The following theorem of Deligne and Ribet \cite{DeligneRibet} gives the fundamental integrality
property of this element.

\begin{theorem}[Deligne--Ribet]
\label{thm:DR-integrality-intro}
Assume in addition that \(k\) is totally real. Then, for any prime number \(p\)
and any integer \(n\geq 1\), one has
\[
\operatorname{Ann}_{\mathbb Z_p[G]}
\bigl(K_{2n-1}(O_K)_{\mathrm{tors}}\otimes_{\mathbb Z}\mathbb Z_p\bigr)
\cdot
\theta_{K/k,S}(1-n)
\subseteq
\mathbb Z_p[G].
\]
\end{theorem}

\noindent This integrality theorem is the input for the Coates--Sinnott conjecture \cite{Coates-Sinnott}, which
predicts that the same equivariant special values annihilate even
\(K\)-groups.

\begin{conject}[Coates--Sinnott]
\label{conj:CS-intro}
Suppose that \(K/k\) is a finite abelian extension of totally real number
fields and that \(n\geq 2\) is even. Then, for every prime number \(p\), one has
\[
\operatorname{Ann}_{\mathbb Z_p[G]}
\bigl(K_{2n-1}(O_K)_{\mathrm{tors}}\otimes_{\mathbb Z}\mathbb Z_p\bigr)
\cdot
\theta_{K/k,S}(1-n)
\subseteq
\operatorname{Ann}_{\mathbb Z_p[G]}
\bigl(K_{2n-2}(O_K)\otimes_{\mathbb Z}\mathbb Z_p\bigr).
\]
\end{conject}

\noindent An unconditional proof of the Coates--Sinnott conjecture for odd prime numbers $p$ was given by Johnston and Nickel
\cite{JohnstonNickel}.
The statements above are formulated in terms of values of equivariant
\(L\)-functions. However, in natural parity situations these values vanish. For
example, if \(K/k\) is totally real and \(n\) is odd, then
\[
\operatorname{ord}_{s=1-n}L_{K/k,S}(s,\chi)>0
\]
for every \(\chi\in\widehat G\), and hence
\[
\theta_{K/k,S}(1-n)=0.
\]
In this case the Coates--Sinnott conjecture gives no non-trivial information.
It is then natural to ask whether an analogous statement can be formulated
using the first non-zero terms of the functions
\(L_{K/k,S}(s,\chi)\) at \(s=1-n\). This is also consistent with the form of
the analytic class number formula, the Lichtenbaum conjecture, and the
equivariant Tamagawa number conjecture, where leading terms rather than merely
values play the central role.

\medskip

\noindent The present paper treats the rank-one case. In this situation the relevant
leading terms are first derivatives. We 
consider the first-derivative component
\[
\theta'_{K/k,S}(1-n)
=
\sum_{r(\chi,n)=1}
L'_{K/k,S}(1-n,\chi^{-1})e_\chi .
\]
Unlike the value \(\theta_{K/k,S}(1-n)\), this element is not expected to be
rational in the group ring. Fixing an odd prime number \(p\) and using the inverse of the Beilinson regulator on the
rank-one component, we associate to it an element 
\[
\Theta'_{K/k,S}(1-n)
\in
\mathbb C_pK_{2n-1}(O_K).
\]
Assuming the equivariant Tamagawa number conjecture, this element is in fact
defined over \(\mathbb Q_p\). Our first main result is the following
first-derivative analogue of the theorem of Deligne and Ribet (see Theorem
\ref{Thm:annihilated-integrality-of-Theta-prime} below).

\begin{theorem}
\label{Main_Theo_01_intro}
One has
\[
\operatorname{Ann}_{\mathbb Z_p[G]}
\bigl(K_{2n-1}(O_K)_{\mathrm{tors}}\otimes_{\mathbb Z}\mathbb Z_p\bigr)
\cdot
\Theta'_{K/k,S}(1-n)
\subseteq
K_{2n-1}(O_K)_{\mathrm{tf}}\otimes_\mathbb{Z} \mathbb Z_p.
\]
\end{theorem}

\medskip

\noindent We then prove a corresponding first-derivative analogue of the
Coates--Sinnott conjecture. Since
\(\Theta'_{K/k,S}(1-n)\) belongs to the rationalized odd \(K\)-group, we attach
to it a fractional \(\mathbb Z_p[G]\)-ideal
\[
\mathcal J(\Theta'_{K/k,S}(1-n))
\subseteq
\mathbb Q_p[G].
\]
It is generated by evaluating at \(\Theta'_{K/k,S}(1-n)\) the rational
extensions of integral \(\mathbb Z_p[G]\)-homomorphisms (see the beginning of
Section \ref{Coates_Sinnott_Section} for the precise definition). The second
main result is the following; see Theorem \ref{Main_Theo_02} below.

\begin{theorem}
\label{Main_Theo_02_intro}
One has
\[
\operatorname{Ann}_{\mathbb Z_p[G]}
\bigl(K_{2n-1}(O_K)_{\mathrm{tors}}\otimes_{\mathbb Z}\mathbb Z_p\bigr)
\cdot
\mathcal J(\Theta'_{K/k,S}(1-n))
\subseteq
\operatorname{Ann}_{\mathbb Z_p[G]}
\bigl(K_{2n-2}(O_{K,S})\otimes_\mathbb{Z} \mathbb Z_p\bigr).
\]
\end{theorem}

\medskip

\noindent We finally point out that the use of the equivariant Tamagawa number conjecture
also makes the results applicable beyond the classical totally real setting.
Indeed, the statements above are formulated for arbitrary finite abelian
extensions \(K/k\); they become unconditional in every case where the relevant
case of the ETNC is known. This includes, for example, many cases with \(k\)
totally real or imaginary quadratic. We recall the precise known cases needed
for this paper in Subsection \ref{ETNCnegative-weight}.

\medskip

\noindent The paper is organized as follows. Sections \ref{Notations} and \ref{preliminaries}
recall the necessary notation and background. Section \ref{The rank d component} is the main body of the paper where we introduce the relevant necessary ingredients of the proof. Subsection \ref{subsect4.2}
proves the first-derivative integrality theorem, and Subsection
\ref{Coates_Sinnott_Section} proves the Coates--Sinnott-type annihilation
result.

\section{Notation}\label{Notations}

\begin{itemize}
\item Let \(K/k\) be a finite abelian extension of number fields with Galois
group \(G\), and let \(S\) be a finite set of places of \(k\) containing all
archimedean places. We write \(O_{K,S}\) for the ring of \(S_K\)-integers of
\(K\), where \(S_K\) denotes the set of places of \(K\) lying above \(S\).

Starting from Subsection \ref{Sect_3_5}, whenever a prime number \(p\) is fixed,
we shall assume in addition that \(S\) contains all places of \(k\) which ramify
in \(K/k\), as well as all places of \(k\) lying above \(p\).

\item Let \(p\) be a prime number. If \(A\) is any \(\mathbb Z_p[G]\)-module, we
let
\[
A^\vee:=\operatorname{Hom}_{\mathbb Z_p}(A,\mathbb Q_p/\mathbb Z_p)
\]
denote the Pontryagin dual of \(A\), endowed with the contragredient
\(G\)-action.

\item We let \(\#\) denote the involution of \(\mathbb Z_p[G]\) induced by
\(g\mapsto g^{-1}\). If \(A\) is any \(\mathbb Z_p[G]\)-module, we define
\(A^\#\) to be the \(\mathbb Z_p[G]\)-module with the same underlying abelian
group as \(A\), on which \(\mathbb Z_p[G]\) acts via this involution.

\item If \(A\) is any \(\mathbb Z\)-module, we let \(A_{\mathrm{tors}}\) denote
its \(\mathbb Z\)-torsion submodule and \(A_{\mathrm{tf}}\) its
\(\mathbb Z\)-torsion-free quotient.

\item If \(A\) is a \(\mathbb Z\)-module and \(R\) is a \(\mathbb Z\)-algebra,
we may write
\[
RA:=A\otimes_{\mathbb Z}R.
\]
Similarly, if \(A\) is a \(\mathbb Z_p\)-module and \(R\) is a
\(\mathbb Z_p\)-algebra, we may write
\[
RA:=A\otimes_{\mathbb Z_p}R.
\]
The base ring of the tensor product will be clear from the context.
\end{itemize}
\section{Algebraic and Cohomological Framework}\label{preliminaries}
\subsection{Quillen's Algebraic \texorpdfstring{$K$}{K}-theory}
We briefly recall some well-known properties of Quillen's higher algebraic
\(K\)-groups associated with \(O_{K,S}\).

\begin{theorem}[Borel]
Let \(n \geq 2\) be an integer. Then the even \(K\)-groups
\(K_{2n-2}(O_{K,S})\) are finite, while the odd \(K\)-groups
\[
K_{2n-1}(O_{K,S}) = K_{2n-1}(O_K),
\]
where \(O_K\) denotes the ring of integers of \(K\), are independent of the
choice of \(S\), as long as \(S\) contains the archimedean places, and are
finitely generated over \(\mathbb Z\). Moreover, their \(\mathbb Z\)-rank is
given by
\[
\operatorname{rk}_{\mathbb Z}\bigl(K_{2n-1}(O_{K,S})\bigr)=d_n:=
\begin{cases}
r_1(K)+r_2(K), & \text{if \(n\) is odd},\\
r_2(K), & \text{if \(n\) is even},
\end{cases}
\]
where \(r_1(K)\) and \(r_2(K)\) denote respectively the number of real and
complex places of \(K\).
\end{theorem}

\medskip

\noindent Fix a rational prime \(p\). For any integer \(i \geq 0\), we write
\[
H^i_{\mathrm{\acute{e}t}}\bigl(\operatorname{Spec}(O_{K,S}),\mathbb Z_p(n)\bigr)
\]
for the corresponding étale cohomology groups. Assume now that \(S\)
contains also all places of \(k\) lying above \(p\). The Quillen--Lichtenbaum
conjecture, proved by Voevodsky \cite{Voevodsky}, Rognes--Weibel
\cite{RognesWeibel}, and Kahn \cite{Kahn}, identifies algebraic \(K\)-theory
with étale cohomology in the following sense: for \(i=1,2\), one has canonical
isomorphisms
\[
\mathbb Z_pK_{2n-i}(O_{K,S})
\cong
H^i_{\mathrm{\acute{e}t}}(\operatorname{Spec}(O_{K,S}),\mathbb Z_p(n))
\]
if \(p\) is odd, and also if \(p=2\) and \(K\) is totally imaginary. In the case
\(p=2\), this follows from the comparison theorem for the Chern class
homomorphism
\[
\mathrm{ch}^K_{n,i}:
\mathbb Z_pK_{2n-i}(O_{K,S})
\longrightarrow
H^i_{\mathrm{\acute{e}t}}(\operatorname{Spec}(O_{K,S}),\mathbb Z_p(n)),
\]
whose kernel and cokernel are controlled by the contribution of the real
places; see \cite{RognesWeibel,Kahn}.
\subsection{The Beilinson regulator map}\label{Sec2.2}

Let \(n \geq 2\) be an integer. Define
\[
Y_{K,n}:=
\bigoplus_{\sigma \in \operatorname{Hom}(K,\mathbb C)}
(2\pi i)^{n-1}\mathbb Z,
\]
where \(\operatorname{Hom}(K,\mathbb C)\) denotes the set of complex embeddings
of \(K\). Let \(c\) denote complex conjugation on \(\mathbb C\). This latter acts on \(Y_{K,n}\) by sending the summand indexed by \(\sigma\) to
the summand indexed by \(c\circ \sigma\), together with the sign induced by
\[
\overline{(2\pi i)^{n-1}}=(-1)^{n-1}(2\pi i)^{n-1}.
\]
We denote by \(Y_{K,n}^{+}\) the submodule of invariants for this action.
\noindent The Beilinson regulator map associated with \(K\) and \(n\) is the homomorphism
\[
r^B_{K,n}:K_{2n-1}(O_K)\longrightarrow \mathbb R Y_{K,n}^{+},
\]
as constructed for instance in \cite{Rapoport}. This map is
\(\mathbb Z[G]\)-linear with respect to the natural action of \(G\), its image
is a full \(\mathbb Z\)-lattice in the real vector space
\(\mathbb R Y_{K,n}^{+}\), and its kernel coincides with the torsion subgroup
\(K_{2n-1}(O_K)_{\mathrm{tors}}\).
Consequently, after extension of scalars, \(r^B_{K,n}\) induces an isomorphism
of \(\mathbb C[G]\)-modules
\[
\widetilde r^B_{K,n}:
\mathbb C K_{2n-1}(O_K)
\xrightarrow{\sim}
\mathbb C Y_{K,n}^{+}.
\]
\subsection{Equivariant Artin \texorpdfstring{$L$}{L}-functions}\label{Sec2.3}
In this subsection we assume  that \(S\) contains all places ramified in \(K/k\) (in addition to the archimedean
places \(S_\infty\) which $S$ is always assumed to contain). Denote by
\[
\widehat G:=\operatorname{Hom}(G,\mathbb C^\times)
\]
the group of complex-valued characters of \(G\).
For a character \(\chi\in\widehat G\) and a (possibly empty) finite set \(\Sigma\) of places of \(k\)
disjoint from \(S\), the \(S\)-truncated, \(\Sigma\)-modified Artin \(L\)-function
attached to \(\chi\) is defined, for \(\operatorname{Re}(s)>1\), by
\[
L_{K/k,S,\Sigma}(s,\chi)
=
\prod_{\mathfrak p\notin S}
\bigl(1-\chi(\sigma_{\mathfrak p})\operatorname{N}\mathfrak p^{-s}\bigr)^{-1}
\prod_{\mathfrak p\in \Sigma}
\bigl(1-\chi(\sigma_{\mathfrak p})\operatorname{N}\mathfrak p^{1-s}\bigr),
\]
where \(\sigma_{\mathfrak p}\) denotes the Frobenius element at
\(\mathfrak p\). This function admits a meromorphic continuation to
\(\mathbb C\), with at most a simple pole at \(s=1\), occurring only when
\(\chi\) is the trivial character. When \(\Sigma=\varnothing\), we simply write
\(L_{K/k,S}(s,\chi)\).

\noindent The functional equation of these \(L\)-functions implies, see for instance
\cite[Proof of Lem.~6.14]{GreitherPopescu}, that for any \(n\geq 2\) one has
\begin{equation}\label{eq:35}
r(\chi,n):=\operatorname{ord}_{s=1-n}L_{K/k,S,\Sigma}(s,\chi)
=
\begin{cases}
r_2(k)+a(\chi)^+, & \text{if \(n\) is odd},\\[4pt]
r_2(k)+a(\chi)^-, & \text{if \(n\) is even},
\end{cases}
\end{equation}
where \(r_2(k)\) denotes the number of complex places of \(k\), and if we write
\(S_{\infty,\mathbb R}(k)\) for the set of real archimedean places of \(k\),
and \(D_v\) for the decomposition subgroup at a place \(v\), then
\[
a(\chi)^+
:=
\#\{\,v\in S_{\infty,\mathbb R}(k)\mid \chi(D_v)=1\,\},
\qquad
a(\chi)^-
:=
\#\{\,v\in S_{\infty,\mathbb R}(k)\mid \chi(D_v)\neq 1\,\}.
\]
Thus
\[
a(\chi)^+ + a(\chi)^- = r_1(k),
\]
where \(r_1(k)\) denotes the number of real places of \(k\).
For \(\chi\in\widehat G\), let
\[
e_\chi:=
\frac{1}{|G|}\sum_{\sigma\in G}\chi(\sigma)\sigma^{-1}
\]
be the associated idempotent in \(\mathbb C[G]\). The equivariant
\(L\)-function, or Stickelberger element, attached to \((K/k,S,\Sigma)\) is defined
by
\[
\theta_{K/k,S,\Sigma}(s)
:=
\sum_{\chi\in\widehat G}L_{K/k,S,\Sigma}(s,\chi^{-1})e_\chi.
\]
These elements interpolate the values of Artin \(L\)-functions and may be
viewed as equivariant analogues of the Dedekind zeta function. For any integer \(n\geq 2\), we let
\[
\theta^*_{K/k,S,\Sigma}(1-n)
\]
denote the leading term of the function \(\theta_{K/k,S,\Sigma}(s)\) at
\(s=1-n\).
\subsection{Cohomology with compact support and perfect complexes}\label{Sect_3_5}
\textit{We assume from now on that the set  \(S\)
contains the archimedean places, the places which ramify in $K/k$ and the places of \(k\) lying above \(p\).}

\noindent For any integer \(n\geq 2\),
we write
\[
R\Gamma_{\mathrm{\acute{e}t}}
\bigl(\operatorname{Spec}(O_{K,S}),\mathbb Z_p(n)\bigr)
\]
for the corresponding complex of étale cohomology. For \(w\in S_K\), we write
\(R\Gamma(K_w,\mathbb Z_p(n))\) for the corresponding local cohomology complex.
These complexes are regarded as objects in the derived category  $D(\mathbb Z_p[G])$ of complexes of
\(\mathbb Z_p[G]\)-modules.

\noindent Recall that, in this setting, the complex of cohomology with compact support is
defined by
\[
R\Gamma_c(\operatorname{Spec}(O_{K,S}),\mathbb Z_p(n))
:=
\operatorname{Cone}
\left(
R\Gamma_{\mathrm{\acute{e}t}}(\operatorname{Spec}(O_{K,S}),\mathbb Z_p(n))
\longrightarrow
\bigoplus_{w\in S_K}R\Gamma(K_w,\mathbb Z_p(n))
\right)[-1],
\]
see for instance \cite[Eq.~(3)]{BurnsFlach2}.
If \(A\) is a \(\mathbb Z_p[G]\)-module, we write
\[
A^*:=\operatorname{Hom}_{\mathbb Z_p}(A,\mathbb Z_p).
\]
Similarly, if \(C^\bullet\) is a complex of \(\mathbb Z_p[G]\)-modules, we write
\[
(C^\bullet)^*
:=
\operatorname{RHom}_{\mathbb Z_p}(C^\bullet,\mathbb Z_p).
\]
Let \(\mathcal D^p(\mathbb Z_p[G])\) denote the full subcategory of the derived
category consisting of perfect complexes of \(\mathbb Z_p[G]\)-modules. 
For any integer $n\geq 2$, we set 
\[\mathcal{C}_{K,S}^\bullet:=R\Gamma_c(\operatorname{Spec}(O_{K,S}),\mathbb Z_p(1-n))^*[-2],
\]
and recall the following result.
\begin{proposition}\label{PROP-2-4}
One has
\[
\mathcal{C}_{K,S}^\bullet
\in
\mathcal D^p(\mathbb Z_p[G]).
\]
Moreover, if either \(p\neq 2\) or \(K\) is totally imaginary, then the complex
\(\mathcal{C}_{K,S}^\bullet\) is acyclic outside degrees
\(0\) and \(1\). Its cohomology groups are described as follows:
\[
H^0\bigl(\mathcal{C}_{K,S}^\bullet)
=
H^1_{\mathrm{\acute{e}t}}
\bigl(\operatorname{Spec}(O_{K,S}),\mathbb Z_p(n)\bigr),
\]
and there is an exact sequence of \(\mathbb Z_p[G]\)-modules
\[
0
\longrightarrow
H^2_{\mathrm{\acute{e}t}}
\bigl(\operatorname{Spec}(O_{K,S}),\mathbb Z_p(n)\bigr)
\longrightarrow
H^1(\mathcal{C}_{K,S}^\bullet)
\longrightarrow
\mathbb{Z}_pY_{K,n}^{+}
\longrightarrow
0.
\]
\end{proposition}

\begin{proof}
The fact that
\[
R\Gamma_c(\operatorname{Spec}(O_{K,S}),\mathbb Z_p(1-n))^*[-2]
\in
\mathcal D^p(\mathbb Z_p[G])
\]
follows from the perfection of the compactly supported cohomology complex
\(R\Gamma_c(\operatorname{Spec}(O_{K,S}),\mathbb Z_p(1-n))\), see
\cite[Proposition~1.20]{BurnsFlach}, together with the fact that the
\(\mathbb Z_p\)-linear dual of a perfect complex of \(\mathbb Z_p[G]\)-modules
is again perfect; see, for example, \cite[\S 3]{BunsG1}.

\noindent We next compute its cohomology. By Artin--Verdier duality at the level of
complexes, one has a distinguished triangle (see \cite[Proposition~4.1]{BurnsFlach2})
\[
\bigoplus_{w\in S_\infty(K)}
R\Gamma_\Delta(K_w,\mathbb Z_p(1-n))^*[-3]
\longrightarrow
R\Gamma_{\mathrm{\acute{e}t}}(\operatorname{Spec}(O_{K,S}),\mathbb Z_p(n))
\longrightarrow
R\Gamma_c(\operatorname{Spec}(O_{K,S}),\mathbb Z_p(1-n))^*[-3].
\]
 Here \(S_\infty(K)\) denotes the set
of archimedean places of \(K\), and, for \(w\in S_\infty(K)\), the complex
\(R\Gamma_\Delta(K_w,\mathbb Z_p(1-n))\) is defined by (see \cite[Eq. (82)]{BurnsFlach2})
\[
R\Gamma_\Delta(K_w,\mathbb Z_p(1-n))
:=
\operatorname{Cone}
\left(
R\Gamma(K_w,\mathbb Z_p(1-n))
\longrightarrow
R\Gamma_{\mathrm{Tate}}(K_w,\mathbb Z_p(1-n))
\right)[-1],
\]
where \(R\Gamma_{\mathrm{Tate}}(K_w,\mathbb Z_p(1-n))
\) denotes the complex of modified Tate cohomology. Assume now that either \(p\neq 2\) or \(K\) is totally imaginary. For an
archimedean place \(w\), the local Tate cohomology groups occurring in
\(R\Gamma_{\mathrm{Tate}}(K_w,\mathbb Z_p(1-n))\) vanish under this assumption.
Hence
\[
R\Gamma_\Delta(K_w,\mathbb Z_p(1-n))
\simeq
R\Gamma(K_w,\mathbb Z_p(1-n)).
\]
The latter complex is concentrated in degree \(0\), with cohomology
\[
H^0(K_w,\mathbb Z_p(1-n)).
\]
As \(w\) runs through the archimedean places of \(K\), these groups form the
\(\mathbb Z_p\)-linear dual of \(\mathbb{Z}_pY_{K,n}^{+}\).
More precisely, the natural pairing between the twists \(1-n\) and \(n-1\)
induces an isomorphism
\[
\bigoplus_{w\in S_\infty(K)}
H^0(K_w,\mathbb Z_p(1-n))
\cong
\operatorname{Hom}_{\mathbb Z_p}
\left(
\mathbb{Z}_pY_{K,n}^{+},
\mathbb Z_p
\right).
\]
Therefore,
\[
\left(
\bigoplus_{w\in S_\infty(K)}
R\Gamma_\Delta(K_w,\mathbb Z_p(1-n))
\right)^*[-2]
\simeq
\left(\mathbb{Z}_pY_{K,n}^{+}\right)[-2].
\]
Consequently, the Artin--Verdier distinguished triangle above yields
\begin{equation}\label{étale_to_compact-support}
R\Gamma_{\mathrm{\acute{e}t}}(\operatorname{Spec}(O_{K,S}),\mathbb Z_p(n))[1]
\longrightarrow
R\Gamma_c(\operatorname{Spec}(O_{K,S}),\mathbb Z_p(1-n))^*[-2]
\longrightarrow
\left(\mathbb{Z}_pY_{K,n}^{+}\right)[-1].
\end{equation}
The computation of the cohomology of $\mathcal{C}_{K,S}^\bullet$ follows from this distinguished triangle.\qed
\end{proof}
\section{The rank $d$ component and some key results}\label{The rank d component}

For every prime number \(p\), we fix, once and for all, an isomorphism
\(\mathbb C\simeq \mathbb C_p\). Via this choice, we shall regard complex
Artin \(L\)-values and the Beilinson regulator as \(p\)-adic objects. In
particular, by abuse of notation, we shall again write
\[
\widetilde r^B_{K,n}:
\mathbb C_p K_{2n-1}(O_K)
\xrightarrow{\sim}
\mathbb C_p Y_{K,n}^{+}
\]
for the induced isomorphism of \(\mathbb C_p[G]\)-modules. Equivalently, this
is obtained from the Beilinson regulator map
\[
r^B_{K,n}:K_{2n-1}(O_K)\longrightarrow \mathbb C_p Y_{K,n}^{+}
\]
after extension of scalars to \(\mathbb C_p\).
For any integer \(d\geq 0\), we consider the idempotent
\[
\mathfrak e^d_{K,n}
:=
\sum_{\substack{\chi\in\widehat G\\ r(\chi,n)=d}}e_\chi
\in
\mathbb C_p[G],
\]
where \(r(\chi,n)\) is defined in \eqref{eq:35}. 
For a subgroup \(H\) of \(G\), we put
\(\
N_H:=\sum_{\sigma\in H}\sigma
\).
If \(v\in S_{\infty,\mathbb R}(k)\), then \(D_v\) is either trivial or of
order \(2\). This defines two idempotents of \(\mathbb Z[1/2][G]\) by
\[
\varepsilon_v^{+}:=\frac{N_{D_v}}{|D_v|},
\qquad
\varepsilon_v^{-}:=1-\frac{N_{D_v}}{|D_v|}.
\]
Thus \(\varepsilon_v^{+}\) is the projector onto the characters which are
trivial on \(D_v\), while \(\varepsilon_v^{-}\) is the projector onto the
characters which are non-trivial on \(D_v\).

\begin{proposition}
Let \(d\geq 0\) be an integer and put
\(
q:=d-r_2(k)
\). Then the following statements hold.
\begin{itemize}
\item If \(q<0\) or \(q>r_1(k)\), then:
\(
\;\mathfrak e^d_{K,n}=0.
\)
\item If \(0\leq q\leq r_1(k)\) and \(n\) is odd, then
\[
\mathfrak e^d_{K,n}
=
\sum_{\substack{A\subseteq S_{\infty,\mathbb R}(k)\\ |A|=q}}
\left(
\prod_{v\in A}\varepsilon_v^{+}
\right)
\left(
\prod_{v\in S_{\infty,\mathbb R}(k)\setminus A}\varepsilon_v^{-}
\right).
\]
\item If \(0\leq q\leq r_1(k)\) and \(n\) is even, then
\[
\mathfrak e^d_{K,n}
=
\sum_{\substack{A\subseteq S_{\infty,\mathbb R}(k)\\ |A|=q}}
\left(
\prod_{v\in A}\varepsilon_v^{-}
\right)
\left(
\prod_{v\in S_{\infty,\mathbb R}(k)\setminus A}\varepsilon_v^{+}
\right).
\]
\item In particular, if \(p\neq 2\), then:
\(
\;\mathfrak e^d_{K,n}\in \mathbb Z_p[G].
\)
\end{itemize}
\end{proposition}

\begin{proof}
Let \(v\in S_{\infty,\mathbb R}(k)\). Since \(K/k\) is abelian, the subgroup
\(D_v\) is well-defined independently of the choice of a place of \(K\) above
\(v\). For any character \(\chi\in\widehat G\), one has
\[
e_\chi\varepsilon_v^{+}
=
\begin{cases}
e_\chi, & \text{if \(\chi(D_v)=1\),}\\
0, & \text{otherwise,}
\end{cases} \qquad \text{and} \qquad
e_\chi\varepsilon_v^{-}
=
\begin{cases}
0, & \text{if \(\chi(D_v)=1\),}\\
e_\chi, & \text{otherwise.}
\end{cases}
\]
Indeed,
\(
e_\chi\frac{N_{D_v}}{|D_v|}
=
\frac{1}{|D_v|}\sum_{\tau\in D_v}\chi(\tau)e_\chi,
\)
which is equal to \(e_\chi\) if \(\chi\) is trivial on \(D_v\), and to \(0\)
otherwise.
Assume first that \(n\) is odd. By \eqref{eq:35}, one has
\[
r(\chi,n)=r_2(k)+a(\chi)^+.
\]
Thus the condition \(r(\chi,n)=d\) is equivalent to: 
\(
a(\chi)^+=q.
\)
For a fixed subset \(A\subseteq S_{\infty,\mathbb R}(k)\), the idempotent
\[
\left(
\prod_{v\in A}\varepsilon_v^{+}
\right)
\left(
\prod_{v\in S_{\infty,\mathbb R}(k)\setminus A}\varepsilon_v^{-}
\right)
\]
projects precisely onto those characters \(\chi\) such that
\[
\chi(D_v)=1 \quad \text{for \(v\in A\)}
\qquad \text{and}\qquad
\chi(D_v)\neq 1 \quad \text{for \(v\notin A\)}.
\]
Summing over all subsets \(A\) of cardinality \(q\) therefore gives exactly the
sum of the idempotents \(e_\chi\) for which \(r(\chi,n)=d\). This proves the
formula for \(n\) odd.
\vskip 5pt

\noindent The case where \(n\) is even is identical, except that \eqref{eq:35} gives
\[
r(\chi,n)=r_2(k)+a(\chi)^-.
\]
Thus one has to project onto the characters which are non-trivial on \(D_v\) at
exactly \(q\) real places \(v\), which gives the stated formula.

\noindent Finally, for an archimedean real place \(v\), the group \(D_v\) has order
\(1\) or \(2\). Hence, if \(p\neq 2\), both \(\varepsilon_v^{+}\) and
\(\varepsilon_v^{-}\) belong to \(\mathbb Z_p[G]\). The displayed formulae then
show that
\(
\mathfrak e^d_{K,n}\in \mathbb Z_p[G].
\)\qed
\end{proof}
\begin{lemma}
\label{lem:Y-free-over-ed}
With the same notations as before, let \(p\) be an odd prime and \(d\geq 0\). 
Then
there is an isomorphism of \(\mathfrak e_{K,n}^d\mathbb Z_p[G]\)-modules
\[
\mathfrak e_{K,n}^d
\bigl(\mathbb{Z}_pY_{K,n}^{+}\bigr)
\cong
\bigl(\mathfrak e_{K,n}^d\mathbb Z_p[G]\bigr)^d.
\]
In particular, if \(\mathfrak e_{K,n}^d=0\), then both sides are zero.
\end{lemma}
\begin{proof}
By the previous proposition, since $p$ is odd we have 
\[
\mathfrak e_{K,n}^d\in\mathbb Z_p[G].
\]
We first recall the structure of the archimedean module
\(\mathbb{Z}_pY_{K,n}^{+}\). For each archimedean place
\(v\) of \(k\), let \(Y_{K,n,v}^{+}\) denote the direct summand of
\(Y_{K,n}^{+}\) corresponding to the embeddings of \(K\) inducing the place
\(v\). Thus
\[
Y_{K,n}^{+}
=
\bigoplus_{v\in S_\infty(k)}Y_{K,n,v}^{+}.
\]
We distinguish the following cases. If \(v\) is a complex place of \(k\), then
\[
Y_{K,n,v}^{+}\otimes_{\mathbb Z}\mathbb Z_p
\cong
\mathbb Z_p[G].
\]
Indeed, the embeddings above a complex place occur in conjugate pairs, and the
``plus'' condition leaves one free copy of \(\mathbb Z_p[G]\).
Now let \(v\) be a real place of \(k\). If \(v\) splits completely in \(K/k\),
then \(D_v=\{1\}\). In this case complex conjugation acts on the corresponding
summand only through the sign
\(
\overline{(2\pi i)^{n-1}}=(-1)^{n-1}(2\pi i)^{n-1}.
\)
Hence
\[
Y_{K,n,v}^{+}\otimes_{\mathbb Z}\mathbb Z_p
\cong
\begin{cases}
\mathbb Z_p[G], & \text{if \(n\) is odd},\\
0, & \text{if \(n\) is even}.
\end{cases}
\]
Finally, suppose that \(v\) is real and does not split completely in \(K/k\).
Then \(D_v\) has order \(2\). Let \(\tau_v\) denote its non-trivial element, and
let \(\xi_{v,n}\) be the character of \(D_v\) defined by
\[
\xi_{v,n}(\tau_v)=(-1)^{n-1}.
\]
Then there is an isomorphism of \(\mathbb Z_p[G]\)-modules
\[
Y_{K,n,v}^{+}\otimes_{\mathbb Z}\mathbb Z_p
\cong
\mathbb Z_p[G]\otimes_{\mathbb Z_p[D_v]}\mathbb Z_p(\xi_{v,n}),
\]
where \(\mathbb Z_p(\xi_{v,n})\) denotes the free rank-one
\(\mathbb Z_p\)-module on which \(D_v\) acts through \(\xi_{v,n}\).
Since \(p\) is odd, the order of \(D_v\) is invertible in \(\mathbb Z_p\). In
particular, when \(D_v\) has order \(2\), the module
\(\mathbb Z_p(\xi_{v,n})\) is projective over \(\mathbb Z_p[D_v]\): explicitly,
it is cut out by the idempotent
\[
\frac{1+\xi_{v,n}(\tau_v)\tau_v}{2}
\in
\mathbb Z_p[D_v].
\]
Hence
\[
\mathbb Z_p[G]\otimes_{\mathbb Z_p[D_v]}\mathbb Z_p(\xi_{v,n})
\]
is a direct summand of \(\mathbb Z_p[G]\), and is therefore projective over
\(\mathbb Z_p[G]\).
It follows that every summand
\(Y_{K,n,v}^{+}\otimes_{\mathbb Z}\mathbb Z_p\) is either zero or projective as
a \(\mathbb Z_p[G]\)-module. Therefore
\(
\mathbb{Z}_pY_{K,n}^{+}
\)
is a finitely generated projective \(\mathbb Z_p[G]\)-module. Since
\(\mathfrak e_{K,n}^d\in\mathbb Z_p[G]\), the module
\[
\mathfrak e_{K,n}^d
\bigl(\mathbb{Z}_pY_{K,n}^{+}\bigr)
\]
is a finitely generated projective module over the ring
\(
\mathfrak e_{K,n}^d\mathbb Z_p[G].
\)
We now compute its rank on each local factor of
\(\mathfrak e_{K,n}^d\mathbb Z_p[G]\). First note that, after extending scalars
to \(\mathbb C_p\), the preceding description gives, for every
\(\chi\in\widehat G\),
\[
\dim_{\mathbb C_p}
e_\chi\bigl(\mathbb C_pY_{K,n}^{+}\bigr)
=
r(\chi,n).
\]
Indeed, each complex place of \(k\) contributes one dimension to every
character component. If \(v\) is real, then its contribution is one precisely
when \(\chi(D_v)=1\) and \(n\) is odd, and precisely when \(\chi(D_v)\neq 1\)
and \(n\) is even. Thus the total dimension is
\[
r_2(k)+a(\chi)^+
\qquad \text{
if \(n\) is odd, and}
\qquad 
r_2(k)+a(\chi)^-
\qquad \text{if \(n\) is even}\] which is exactly \(r(\chi,n)\) by \eqref{eq:35}.
Consequently, if \(e_\chi\) occurs in \(\mathfrak e_{K,n}^d\), then
\[
\dim_{\mathbb C_p}
e_\chi\Bigl(
\mathbb C_p\otimes_{\mathbb Z_p}
\mathfrak e_{K,n}^d
\bigl(\mathbb{Z}_pY_{K,n}^{+}\bigr)
\Bigr)
=
d.
\]
On the other hand, for every such character \(\chi\), one has
\[
e_\chi
\bigl(
\mathbb C_p\otimes_{\mathbb Z_p}
\mathfrak e_{K,n}^d\mathbb Z_p[G]
\bigr)
\cong
\mathbb C_p.
\]
Therefore
\[
\mathbb C_p\otimes_{\mathbb Z_p}
\mathfrak e_{K,n}^d
\bigl(\mathbb{Z}_pY_{K,n}^{+}\bigr)
\cong
\left(
\mathbb C_p\otimes_{\mathbb Z_p}
\mathfrak e_{K,n}^d\mathbb Z_p[G]
\right)^d
\]
as \(\mathbb C_p[G]\)-modules.
It remains to descend this equality of generic ranks to an integral statement.
The ring
\[
\mathfrak e_{K,n}^d\mathbb Z_p[G]
\]
is a finite commutative semilocal \(\mathbb Z_p\)-algebra. Hence it decomposes
as a finite product of local rings
\[
\mathfrak e_{K,n}^d\mathbb Z_p[G]
=
\prod_i A_i.
\]
The corresponding components of
\[
\mathfrak e_{K,n}^d
\bigl(\mathbb{Z}_pY_{K,n}^{+}\bigr)
\]
are finitely generated projective modules over the local rings \(A_i\). Since
every finitely generated projective module over a local ring is free, each
component is free over the corresponding \(A_i\).
The computation after scalar extension to \(\mathbb C_p\) shows that the rank
of each of these free \(A_i\)-modules is equal to \(d\). Hence each component is
isomorphic to \(A_i^d\). Taking the product over all local factors gives
\[
\mathfrak e_{K,n}^d
\bigl(\mathbb{Z}_pY_{K,n}^{+}\bigr)
\cong
\bigl(\mathfrak e_{K,n}^d\mathbb Z_p[G]\bigr)^d,
\]
as claimed.\qed
\end{proof}
\subsection{The Equivariant Tamagawa Number Conjecture for $(h^0(\operatorname{Spec}(K))(1-n), \mathbb Z_p[G])$}\label{ETNCnegative-weight}
As before, suppose that $n\geq 2$.
\textit{In what follows we will always assume that $p$ is an odd prime number.} 
Therefore, we adopt  the following Quillen-Lichtenbaum identification in the rest of this work
\[
\mathbb Z_pK_{2n-i}(O_{K,S})
=
H^i_{\mathrm{\acute{e}t}}(\operatorname{Spec}(O_{K,S}),\mathbb Z_p(n))\qquad \text{for}\qquad i=1,2.
\]
\noindent We recall below the statement of the Equivariant Tamagawa Number Conjecture (abbreviated as the ETNC) for the pair $(h^0(\operatorname{Spec}(K))(1-n), \mathbb Z_p[G])$ (see e.g., \cite{Flach} for a more general exposition on this conjecture).\\
Consider the  trivialization defined as the following composition of isomorphisms
\begin{align}\label{Trivialization}
\nabla_{K,S}:\;\mathbb{C}_p\operatorname{det}_{\mathbb{Z}_p[G]}(\mathcal{C}^\bullet_{K,S})&\simeq \operatorname{det}_{\mathbb{C}_p[G]}(\mathbb{C}_pK_{2n-1}(O_{K}))\otimes_{\mathbb{C}_p[G]}\operatorname{det}^{-1}_{\mathbb{C}_p[G]}(\mathbb{C}_pY_{K,n}^+) \nonumber\\
    &\stackrel{\simeq}{\longrightarrow}\operatorname{det}_{\mathbb{C}_p[G]}(\mathbb{C}_pY_{K,n}^+)\otimes_{\mathbb{C}_p[G]}\operatorname{det}^{-1}_{\mathbb{C}_p[G]}(\mathbb{C}_pY_{K,n}^+) \nonumber\\
    &\stackrel{\simeq}{\longrightarrow}\mathbb{C}_p[G],
\end{align}
where the first map is the standard \textit{passage to cohomology} isomorphism (obtained from the description of the cohomology of $\mathcal{C}_{K,S}^\bullet$ in Proposition \ref{PROP-2-4}), the second is induced by the Beilinson regulator map $\widetilde r^B_{K,n}$, and the last one is the standard evaluation map.
The ETNC for the pair $(h^0(\operatorname{Spec}(K))(1-n), \mathbb Z_p[G])$ is formulated as follows (e.g., see \cite[\S 3]{BunsG1})
\begin{conjecture}{\textbf{(The ETNC for $(h^0(\operatorname{Spec}(K))(1-n), \mathbb Z_p[G])$)}.}\label{ConjectETNC} One has
\[\nabla_{K,S}\bigl(\operatorname{det}_{\mathbb Z_p[G]}(\mathcal C^\bullet_{K,S})\bigr)=\theta^{*}_{K/k,S}(1-n)\mathbb{Z}_p[G].
\]
\end{conjecture}
In the rest of the paper, we work under the standing hypothesis that the ETNC
holds for the pair
\[
\bigl(h^0(\operatorname{Spec}(K))(1-n),\mathbb Z_p[G]\bigr).
\]
This hypothesis is known in many cases relevant to the present work. We recall
some of them:
\begin{itemize}
    \item \(K\) is an abelian number field \cite{BunsG1}.
    
    \item \(K/k\) is a finite abelian extension of number fields with Galois
    group \(G\), where \(k\) is imaginary quadratic, and \(p\nmid 6|G|\) is a
    prime number which splits in \(k\) \cite{Leung}.
    
    \item If \(K/k\) is an abelian CM extension of number fields with Galois group
    \(G\) and  \(p\) is odd.
    Then
    for each even \(n\geq 2\), respectively odd \(n\geq 3\), the plus,
    respectively minus, part of the ETNC for
    \(
    \bigl(h^0(\operatorname{Spec}(K))(1-n),\mathbb Z_p[G]\bigr)
    \)
    holds 
   \cite{JohnstonNickel}.
\end{itemize}
\subsection{Twisted Stickelberger-Stark elements in algebraic $K$-theory}\label{subsect4.2}
Recall that $p$ is an odd prime number and fix an integer $n\geq 2$.
Recall also that we write $\theta_{K/k,S}^*(1-n)\in\mathbb{C}_p[G]^\times$ for the leading term 
of the function $\theta_{K/k,S}(s)$ at  $s=1-n$. In what follows, we fix $d=1$ and are interested in the first derivative of  $\theta_{K/k,S}(s)$ at  $s=1-n$:
\[
\theta_{K/k,S}'(1-n)
:=
\mathfrak e^1_{K,n}\theta_{K/k,S}^*(1-n)
=
\sum_{\substack{\chi\in\widehat G\\ r(\chi,n)=1}}
L'_{K/k,S}(1-n,\chi^{-1})e_\chi
\in
(\mathfrak e^1_{K,n}\mathbb C_p[G])^\times.
\]
where $\mathfrak e^1_{K,n}\in\mathbb{Z}_p[G]$ is the idempotent defined in \S \ref{The rank d component} with $d=1$. The Beilinson regulator map \(r^B_{K,n}\) induces an isomorphism
of \(\mathbb C_p[G]\)-modules
\[
\widetilde r^B_{K,n}:
\mathbb C_p K_{2n-1}(O_K)
\xrightarrow{\sim}
\mathbb C_p Y_{K,n}^{+}.
\]
By Lemma \ref{lem:Y-free-over-ed}, we have an identification
\[
\mathfrak e^1_{K,n}\mathbb C_p Y_{K,n}^{+}\cong \mathfrak e^1_{K,n}\mathbb C_p[G].
\]
Therefore, the induced map on the $\mathfrak e^1_{K,n}$-component defines an isomorphism
\[
\widetilde r^B_{K,n}:
\mathfrak e^1_{K,n}\mathbb C_p K_{2n-1}(O_K)
\xrightarrow{\sim}
\mathfrak e^1_{K,n}\mathbb C_p [G].
\]
Inspired by the abelian Stark conjecture for number fields \cite{Stark}, we define the \textit{Twisted Stickelberger-Stark element} for $(K/k, S, n)$ as the unique element $\Theta'_{K/k,S}(1-n)\in\mathbb{C}_pK_{2n-1}(O_K)$ defined by
\[
\Theta'_{K/k,S}(1-n):=(\widetilde r^B_{K,n})^{-1}\big( \theta_{K/k,S}'(1-n)\big).
\]
The key result of this subsection is the following theorem.
\begin{theorem}
\label{Thm:annihilated-integrality-of-Theta-prime}
We have
\[
\operatorname{Ann}_{\mathbb Z_p[G]}
\bigl(\mathbb Z_pK_{2n-1}(O_K)_{\mathrm{tors}}\bigr)
\cdot
\Theta'_{K/k,S}(1-n)
\subseteq
\mathbb Z_pK_{2n-1}(O_K)_{\mathrm{tf}}.
\]
\end{theorem}
Before proving Theorem \ref{Thm:annihilated-integrality-of-Theta-prime}, we need to record some preparatory results. We will do this in the following paragraphs.
\subsubsection{$\Sigma$-cutout}

Let $\Sigma$ be a finite set of finite places of $k$ disjoint from $S$, and
let $\Sigma_K$ denote the set of places of $K$ above $\Sigma$. We suppose again that \(n\geq 2\) and that \(p\) is an odd prime number, and
we let \(D(\mathbb Z_p[G])\) denote the derived category of complexes of
\(\mathbb Z_p[G]\)-modules. Recall that
the
complex of étale $\Sigma$-cohomology is defined in $D(\mathbb Z_p[G])$ by
\[
R\Gamma_{\Sigma}(\operatorname{Spec}(O_{K,S}),\mathbb Z_p(n))
:=
\operatorname{Cone}\left(
R\Gamma_{\mathrm{\acute{e}t}}(\operatorname{Spec}(O_{K,S}),\mathbb Z_p(n))
\longrightarrow
\bigoplus_{w\in\Sigma_K}R\Gamma(\kappa(w),\mathbb Z_p(n))
\right)[-1],
\]
where $\kappa(w)$ is the residue field at $w$. We set
\[
H^i_{\Sigma}(\operatorname{Spec}(O_{K,S}),\mathbb Z_p(n))
:=
H^i\left(
R\Gamma_{\Sigma}(\operatorname{Spec}(O_{K,S}),\mathbb Z_p(n))
\right).
\]
The associated long exact sequence gives
\begin{align*}
0\to&
H^1_{\Sigma}(\operatorname{Spec}(O_{K,S}),\mathbb Z_p(n))
\to
H^1_{\mathrm{\acute{e}t}}(\operatorname{Spec}(O_{K,S}),\mathbb Z_p(n))
\to
\bigoplus_{w\in \Sigma_K}H^1(\kappa(w),\mathbb Z_p(n))
\to
H^2_{\Sigma}(\operatorname{Spec}(O_{K,S}),\mathbb Z_p(n))
\\
\to&
H^2_{\mathrm{\acute{e}t}}(\operatorname{Spec}(O_{K,S}),\mathbb Z_p(n))
\to 0.
\end{align*}
For each $v\in\Sigma$, observe that
\[
\bigoplus_{w\mid v}H^1(\kappa(w),\mathbb Z_p(n))
\simeq
\mathbb Z_p[G]/
\bigl(1-\sigma_v^{-1}\operatorname{N}v^n\bigr)\mathbb Z_p[G],
\]
and
\(
\bigoplus_{w\mid v}H^i(\kappa(w),\mathbb Z_p(n))=0
\)
for $i\neq 1$. Hence the complex
\(
\bigoplus_{w\in\Sigma_K}R\Gamma(\kappa(w),\mathbb Z_p(n))
\)
is isomorphic in $D(\mathbb Z_p[G])$ to the complex
\[
\left[
\bigoplus_{v\in\Sigma}\mathbb Z_p[G]
\stackrel{(1-\sigma_v^{-1}\operatorname{N}v^n)_{v\in\Sigma}}{\longrightarrow}
\bigoplus_{v\in\Sigma}\mathbb Z_p[G]
\right],
\]
where the two terms are placed in degrees $0$ and $1$. In particular,
\(
\bigoplus_{w\in\Sigma_K}R\Gamma(\kappa(w),\mathbb Z_p(n))
\)
is a perfect complex of $\mathbb Z_p[G]$-modules, acyclic outside degree $1$.
\vskip 5pt
\noindent We now construct the $\Sigma$-cutout complex. 
Consider again the distinguished triangle \eqref{étale_to_compact-support}
\[
R\Gamma_{\mathrm{\acute{e}t}}(\operatorname{Spec}(O_{K,S}),\mathbb Z_p(n))[1]
\to
R\Gamma_c(\operatorname{Spec}(O_{K,S}),\mathbb Z_p(1-n))^*[-2]
\to
\left(\mathbb{Z}_pY_{K,n}^{+}\right)[-1]
\to.
\]
By the proof of Lemma \ref{lem:Y-free-over-ed},
\(\mathbb Z_pY_{K,n}^{+}\) is a projective \(\mathbb Z_p[G]\)-module. Since
\(
\bigoplus_{w\in\Sigma_K}R\Gamma(\kappa(w),\mathbb Z_p(n))[1]
\)
is acyclic outside degree \(0\), it follows that
\[
\operatorname{Hom}_{D(\mathbb Z_p[G])}
\left(
\left(\mathbb Z_pY_{K,n}^{+}\right)[-i],
\bigoplus_{w\in\Sigma_K}R\Gamma(\kappa(w),\mathbb Z_p(n))[1]
\right)
=0
\]
for all \(i\geq 1\). Indeed, since the target is quasi-isomorphic to
\[
\bigoplus_{w\in\Sigma_K}H^1(\kappa(w),\mathbb Z_p(n))
\]
placed in degree \(0\), the above group is equal to
\[
\operatorname{Ext}^{i}_{\mathbb Z_p[G]}
\left(
\mathbb{Z}_pY_{K,n}^{+},
\bigoplus_{w\in\Sigma_K}H^1(\kappa(w),\mathbb Z_p(n))
\right),
\]
which vanishes by projectivity.
Applying
\(
\operatorname{Hom}_{D(\mathbb Z_p[G])}
\left(
-,
\bigoplus_{w\in\Sigma_K}R\Gamma(\kappa(w),\mathbb Z_p(n))[1]
\right)
\)
to the preceding distinguished triangle therefore shows that the localization
morphism 
\[
R\Gamma_{\mathrm{\acute{e}t}}(\operatorname{Spec}(O_{K,S}),\mathbb Z_p(n))[1]
\longrightarrow
\bigoplus_{w\in\Sigma_K}R\Gamma(\kappa(w),\mathbb Z_p(n))[1]
\]
admits a unique lift
\[
f\colon
R\Gamma_c(\operatorname{Spec}(O_{K,S}),\mathbb Z_p(1-n))^*[-2]
\longrightarrow
\bigoplus_{w\in\Sigma_K}R\Gamma(\kappa(w),\mathbb Z_p(n))[1].
\]
We define
\[
\mathcal C_{K,S,\Sigma}^{\bullet}
:=
\operatorname{Cone}(f)[-1].
\]
Verdier's $3\times 3$ lemma gives a commutative diagram with
distinguished rows and columns
\[
\begin{tikzcd}
R\Gamma_{\Sigma}(\operatorname{Spec}(O_{K,S}), \mathbb Z_p(n))[1]
\arrow[r]
\arrow[d]
&
\mathcal C_{K,S,\Sigma}^{\bullet}
\arrow[r]
\arrow[d]
&
\left(\mathbb{Z}_pY_{K,n}^{+}\right)[-1]
\arrow[d, "="]
\\
R\Gamma_{\mathrm{\acute{e}t}}(\operatorname{Spec}(O_{K,S}),\mathbb Z_p(n))[1]
\arrow[r]
\arrow[d]
&
R\Gamma_c(\operatorname{Spec}(O_{K,S}),\mathbb Z_p(1-n))^*[-2]
\arrow[r]
\arrow[d, "f"]
&
\left(\mathbb{Z}_pY_{K,n}^{+}\right)[-1]
\arrow[d, "="]
\\
\displaystyle\bigoplus_{w\in\Sigma_K}R\Gamma(\kappa(w),\mathbb Z_p(n))[1]
\arrow[r, "="]
&
\displaystyle\bigoplus_{w\in\Sigma_K}R\Gamma(\kappa(w),\mathbb Z_p(n))[1]
\arrow[r]
&
0.
\end{tikzcd}
\]
Since
\(
R\Gamma_c(\operatorname{Spec}(O_{K,S}),\mathbb Z_p(1-n))^*[-2]
\)
and
\(
\bigoplus_{w\in\Sigma_K}R\Gamma(\kappa(w),\mathbb Z_p(n))[1]
\)
are perfect complexes of $\mathbb Z_p[G]$-modules, the second column shows that
$\mathcal C_{K,S,\Sigma}^{\bullet}$ is perfect. On the other hand, the first
row shows that $\mathcal C_{K,S,\Sigma}^{\bullet}$ is acyclic outside degrees
$0$ and $1$. More precisely, it gives a canonical identification
\[
H^0(\mathcal C_{K,S,\Sigma}^{\bullet})
\simeq
H^1_{\Sigma}(\operatorname{Spec}(O_{K,S}),\mathbb Z_p(n))
\]
and an exact sequence
\[
0\to
H^2_{\Sigma}(\operatorname{Spec}(O_{K,S}),\mathbb Z_p(n))
\to
H^1(\mathcal C_{K,S,\Sigma}^{\bullet})
\to
\mathbb{Z}_pY_{K,n}^{+}
\to 0.
\]
We end this subsection by recording the following torsion-freeness property that will be used below.

\begin{lemma}
\label{lem:Sigma-cohomology-torsion-free}
Assume that $\Sigma$ is non-empty. Then
\(
H^1_{\Sigma}(\operatorname{Spec}(O_{K,S}),\mathbb Z_p(n))
\)
is $\mathbb Z_p$-torsion-free.
\end{lemma}

\begin{proof}
By the long exact sequence defining the $\Sigma$-cohomology complex, one has
\[
H^1_{\Sigma}(\operatorname{Spec}(O_{K,S}),\mathbb Z_p(n))
=
\ker\left(
H^1_{\mathrm{\acute{e}t}}(\operatorname{Spec}(O_{K,S}),\mathbb Z_p(n))
\to
\bigoplus_{w\in\Sigma_K}H^1(\kappa(w),\mathbb Z_p(n))
\right).
\]
Moreover, the $\mathbb Z_p$-torsion subgroup of
\(
H^1_{\mathrm{\acute{e}t}}(\operatorname{Spec}(O_{K,S}),\mathbb Z_p(n))
\)
identifies, via the boundary morphism associated with
\[
0\to \mathbb Z_p(n)\to \mathbb Q_p(n)\to
\mathbb Q_p/\mathbb Z_p(n)\to 0,
\]
with
\(
H^0(K,\mathbb Q_p/\mathbb Z_p(n)).
\)
It is therefore enough to show that the composite map
\[
H^0(K,\mathbb Q_p/\mathbb Z_p(n))
\hookrightarrow
H^1_{\mathrm{\acute{e}t}}(\operatorname{Spec}(O_{K,S}),\mathbb Z_p(n))
\to
\bigoplus_{w\in\Sigma_K}H^1(\kappa(w),\mathbb Z_p(n))
\]
is injective.
Let $w\in\Sigma_K$. Since \(\Sigma\) is disjoint from \(S\), and since \(S\) contains the places of
\(k\) above \(p\), the place \(w\) is non-\(p\)-adic. For such
a place, the cohomology groups
\(
H^0(\kappa(w),\mathbb Q_p(n))
\) and \(
H^1(\kappa(w),\mathbb Q_p(n))
\)
are zero. Indeed, Frobenius acts on $\mathbb Q_p(n)$ through multiplication by
a scalar different from $1$. Hence the long exact sequence associated with
\[
0\to \mathbb Z_p(n)\to \mathbb Q_p(n)\to
\mathbb Q_p/\mathbb Z_p(n)\to 0
\]
gives a canonical isomorphism
\[
H^0(\kappa(w),\mathbb Q_p/\mathbb Z_p(n))
\simeq
H^1(\kappa(w),\mathbb Z_p(n)).
\]
By functoriality of the boundary maps, the composite
\[
H^0(K,\mathbb Q_p/\mathbb Z_p(n))
\hookrightarrow
H^1_{\mathrm{\acute{e}t}}(\operatorname{Spec}(O_{K,S}),\mathbb Z_p(n))
\to
H^1(\kappa(w),\mathbb Z_p(n))
\]
identifies with the restriction map
\[
H^0(K,\mathbb Q_p/\mathbb Z_p(n))
\to
H^0(\kappa(w),\mathbb Q_p/\mathbb Z_p(n))
\]
followed by the preceding isomorphism. 
This restriction map is injective. Indeed, since \(w\nmid p\), the
representation \(\mathbb Q_p/\mathbb Z_p(n)\) is unramified at \(w\), and hence
\[
H^0(K_w,\mathbb Q_p/\mathbb Z_p(n))
=
H^0(\kappa(w),\mathbb Q_p/\mathbb Z_p(n)).
\]
Thus the map in question identifies with the natural injective restriction map
\[
H^0(K,\mathbb Q_p/\mathbb Z_p(n))
\hookrightarrow
H^0(K_w,\mathbb Q_p/\mathbb Z_p(n)).
\]
Since \(\Sigma\) is non-empty, choosing any \(w\in\Sigma_K\) shows that the map
to the direct sum is injective. This proves the claim.\qed
\end{proof}
\subsubsection{Proof of Theorem \ref{Thm:annihilated-integrality-of-Theta-prime}}

Let $\Sigma$ be a non-empty finite set of finite places of $k$ disjoint from
$S$. We first cut out the $\mathfrak e^1_{K,n}$-part of the complex
$\mathcal C_{K,S,\Sigma}^{\bullet}$. Define
\[
\widetilde{\mathcal C}_{K,S,\Sigma}^{\bullet}
:=
\mathcal C_{K,S,\Sigma}^{\bullet}
\otimes_{\mathbb Z_p[G]}^{\mathrm L}
\mathfrak e^1_{K,n}\mathbb Z_p[G].
\]
Since $\mathfrak e^1_{K,n}$ is a central idempotent of $\mathbb Z_p[G]$, this
derived tensor product simply cuts out the $\mathfrak e^1_{K,n}$-component of
$\mathcal C_{K,S,\Sigma}^{\bullet}$. Hence
$\widetilde{\mathcal C}_{K,S,\Sigma}^{\bullet}$ is a perfect complex of
$\mathfrak e^1_{K,n}\mathbb Z_p[G]$-modules, acyclic outside degrees $0$ and
$1$, and
\[
H^0\bigl(\widetilde{\mathcal C}_{K,S,\Sigma}^{\bullet}\bigr)
=
\mathfrak e^1_{K,n}
H^1_{\Sigma}(\operatorname{Spec}(O_{K,S}),\mathbb Z_p(n)).
\]
Moreover, there is an exact sequence
\[
0
\longrightarrow
\mathfrak e^1_{K,n}
H^2_{\Sigma}(\operatorname{Spec}(O_{K,S}),\mathbb Z_p(n))
\longrightarrow
H^1\bigl(\widetilde{\mathcal C}_{K,S,\Sigma}^{\bullet}\bigr)
\longrightarrow
\mathfrak e^1_{K,n}
\bigl(\mathbb{Z}_pY_{K,n}^{+}\bigr)
\longrightarrow
0.
\]
By Lemma \ref{lem:Y-free-over-ed}, one has an isomorphism
\[
\mathfrak e^1_{K,n}
\bigl(\mathbb{Z}_pY_{K,n}^{+}\bigr)
\simeq
\mathfrak e^1_{K,n}\mathbb Z_p[G].
\]
Since $\mathfrak e^1_{K,n}\mathbb Z_p[G]$ is projective over itself, the
previous exact sequence splits. Thus, after choosing a splitting, one has
\[
H^1\bigl(\widetilde{\mathcal C}_{K,S,\Sigma}^{\bullet}\bigr)
\simeq
\mathfrak e^1_{K,n}
H^2_{\Sigma}(\operatorname{Spec}(O_{K,S}),\mathbb Z_p(n))
\oplus
\mathfrak e^1_{K,n}\mathbb Z_p[G].
\]
Let $D(\mathfrak e^1_{K,n}\mathbb Z_p[G])$ denote the derived category of $\mathfrak e^1_{K,n}\mathbb Z_p[G]$-modules and assume next that
\(\mathfrak e^1_{K,n}\neq 0\). We shall use the following representative of
$\widetilde{\mathcal C}_{K,S,\Sigma}^{\bullet}$.

\begin{lemma}
\label{lem:quadratic-representative-Sigma-cutout}
There exists a free $\mathfrak e^1_{K,n}\mathbb Z_p[G]$-module $F$ of finite
rank, with basis $b_1,\ldots,b_m$, and a representative of
$\widetilde{\mathcal C}_{K,S,\Sigma}^{\bullet}$ in
$D(\mathfrak e^1_{K,n}\mathbb Z_p[G])$ of the form
\[
F
\stackrel{\Psi}{\longrightarrow}
F,
\]
where the two terms are placed in degrees $0$ and $1$, respectively. Moreover,
the induced surjection
\(
F
\twoheadrightarrow
H^1\bigl(\widetilde{\mathcal C}_{K,S,\Sigma}^{\bullet}\bigr)
\)
sends $b_1$ to a generator of the free direct summand
\(
\mathfrak e^1_{K,n}\mathbb Z_p[G]
\)
and sends $b_2,\ldots,b_m$ to a set of generators of
\(
\mathfrak e^1_{K,n}
H^2_{\Sigma}(\operatorname{Spec}(O_{K,S}),\mathbb Z_p(n)).
\)
\end{lemma}

\begin{proof}
Choose a finite set of generators of
\(
\mathfrak e^1_{K,n}
H^2_{\Sigma}(\operatorname{Spec}(O_{K,S}),\mathbb Z_p(n))
\)
and add one further generator for the free direct summand
\(
\mathfrak e^1_{K,n}\mathbb Z_p[G]
\)
of
\(
H^1\bigl(\widetilde{\mathcal C}_{K,S,\Sigma}^{\bullet}\bigr).
\)
This gives a finite free $\mathfrak e^1_{K,n}\mathbb Z_p[G]$-module $F$, with
basis $b_1,\ldots,b_m$, and a surjection
\[
F
\twoheadrightarrow
H^1\bigl(\widetilde{\mathcal C}_{K,S,\Sigma}^{\bullet}\bigr)
\]
with the stated properties.
Since $\widetilde{\mathcal C}_{K,S,\Sigma}^{\bullet}$ is perfect and acyclic
outside degrees $0$ and $1$, the standard construction of two-term
representatives for perfect complexes (see e.g., \cite[\S 3]{Burns00}) gives a representative
\[
\mathcal K
\stackrel{\Phi}{\longrightarrow}
F
\]
in $D(\mathfrak e^1_{K,n}\mathbb Z_p[G])$, with $\mathcal K$ finitely generated
and of finite projective dimension over
$\mathfrak e^1_{K,n}\mathbb Z_p[G]$. Its cohomology gives an exact sequence
\[
0
\longrightarrow
H^0\bigl(\widetilde{\mathcal C}_{K,S,\Sigma}^{\bullet}\bigr)
\longrightarrow
\mathcal K
\longrightarrow
F
\longrightarrow
H^1\bigl(\widetilde{\mathcal C}_{K,S,\Sigma}^{\bullet}\bigr)
\longrightarrow
0.
\]
By Lemma \ref{lem:Sigma-cohomology-torsion-free},
\[
H^0\bigl(\widetilde{\mathcal C}_{K,S,\Sigma}^{\bullet}\bigr)
=
\mathfrak e^1_{K,n}
H^1_{\Sigma}(\operatorname{Spec}(O_{K,S}),\mathbb Z_p(n))
\]
is $\mathbb Z_p$-torsion-free. Since the image of $\mathcal K$ in $F$ is a
submodule of the $\mathbb Z_p$-torsion-free module $F$, it follows that
$\mathcal K$ is also $\mathbb Z_p$-torsion-free.
The ring $\mathfrak e^1_{K,n}\mathbb Z_p[G]$ is a direct summand of
$\mathbb Z_p[G]$, and is therefore finite and $\mathbb Z_p$-torsion-free over
$\mathbb Z_p$. Hence its localizations at maximal ideals are one-dimensional
local rings of depth $1$. Since $\mathcal K$ has finite projective dimension
and is $\mathbb Z_p$-torsion-free, the Auslander--Buchsbaum formula \cite[Thm. 19.9]{Eisenbud} shows that
$\mathcal K$ is projective over
$\mathfrak e^1_{K,n}\mathbb Z_p[G]$.

It remains to see that $\mathcal K$ is free of the same rank as $F$. After
tensoring with $\mathbb Q_p$, the local $\Sigma$-terms vanish, and the
regulator isomorphism gives
\[
\mathbb Q_p
H^0\bigl(\widetilde{\mathcal C}_{K,S,\Sigma}^{\bullet}\bigr)
\simeq
\mathbb Q_p
H^1\bigl(\widetilde{\mathcal C}_{K,S,\Sigma}^{\bullet}\bigr).
\]
Therefore the exact sequence above gives
\(
\mathbb Q_p\mathcal K
\simeq
\mathbb Q_pF
\)
as
$\mathfrak e^1_{K,n}\mathbb Q_p[G]$-modules. Since
$\mathfrak e^1_{K,n}\mathbb Z_p[G]$ is semilocal, a finitely generated
projective module whose scalar extension to
$\mathfrak e^1_{K,n}\mathbb Q_p[G]$ is free of the same rank on every factor is
itself free. Hence $\mathcal K$ is free of the same rank as $F$.

\noindent After choosing an isomorphism $\mathcal K\simeq F$, the representative
\(
\mathcal K
\stackrel{\Phi}{\longrightarrow}
F
\)
becomes a representative of the desired form
\[
F
\stackrel{\Psi}{\longrightarrow}
F.
\]
The construction preserves the chosen surjection from the degree-one term onto
$H^1(\widetilde{\mathcal C}_{K,S,\Sigma}^{\bullet})$, and hence the asserted
properties of the basis $b_1,\ldots,b_m$ hold.\qed
\end{proof}
Next, recall that we have fixed an identification, provided by Lemma \ref{lem:Y-free-over-ed},
\[
\mathfrak e_{K,n}^1
\bigl(\mathbb{Z}_pY_{K,n}^{+}\bigr)
\simeq
\mathfrak e_{K,n}^1\mathbb Z_p[G].
\]
Through this identification, the \(\mathfrak e_{K,n}^1\)-component of the
trivialization \(\nabla_{K,S}\) introduced in
\S \ref{ETNCnegative-weight} is given by the composition
\[
\begin{aligned}
\mathfrak e_{K,n}^1\mathbb C_p
\operatorname{det}_{\mathbb Z_p[G]}(\mathcal C_{K,S}^{\bullet})
\stackrel{\pi}{\simeq}&\;
\operatorname{det}_{\mathfrak e_{K,n}^1\mathbb C_p[G]}
\bigl(\mathfrak e_{K,n}^1\mathbb C_pK_{2n-1}(O_K)\bigr)\\
\stackrel{\widetilde r^B_{K,n}}{\simeq}&\;
\operatorname{det}_{\mathfrak e_{K,n}^1\mathbb C_p[G]}
\bigl(\mathfrak e_{K,n}^1\mathbb C_p[G]\bigr)
\simeq
\mathfrak e_{K,n}^1\mathbb C_p[G].
\end{aligned}
\]
Here \(\pi\) denotes the projection, in the passage to cohomology, onto the
determinant of
\[
\mathfrak e_{K,n}^1\mathbb C_pK_{2n-1}(O_K),
\]
after identifying the determinant of the free rank-one quotient
\[
\mathfrak e_{K,n}^1
\bigl(\mathbb C_pY_{K,n}^{+}\bigr)
\simeq
\mathfrak e_{K,n}^1\mathbb C_p[G]
\]
with \(\mathfrak e_{K,n}^1\mathbb C_p[G]\). Since
\[
\mathfrak e_{K,n}^1\mathbb C_pK_{2n-1}(O_K)
\simeq
\mathfrak e_{K,n}^1\mathbb C_p[G],
\]
we identify
\[
\operatorname{det}_{\mathfrak e_{K,n}^1\mathbb C_p[G]}
\bigl(\mathfrak e_{K,n}^1\mathbb C_pK_{2n-1}(O_K)\bigr)
\]
with
\[
\mathfrak e_{K,n}^1\mathbb C_pK_{2n-1}(O_K),
\]
and the second isomorphism above is the Beilinson regulator on the
\(\mathfrak e_{K,n}^1\)-component.
By the ETNC, there exists a generator
\[
\mathfrak z\in
\operatorname{det}_{\mathbb Z_p[G]}(\mathcal C_{K,S}^{\bullet})
\]
such that
\[
\nabla_{K,S}(\mathfrak z)=\theta^*_{K/k,S}(1-n).
\]
After applying \(\mathfrak e_{K,n}^1\), we get
\[
\widetilde r^B_{K,n}\bigl(\pi(\mathfrak e_{K,n}^1\mathfrak z)\bigr)
=
\mathfrak e_{K,n}^1\theta^*_{K/k,S}(1-n)
=
\theta'_{K/k,S}(1-n).
\]
On the other hand, by definition,
\[
\widetilde r^B_{K,n}\bigl(\Theta'_{K/k,S}(1-n)\bigr)
=
\theta'_{K/k,S}(1-n).
\]
Since \(\widetilde r^B_{K,n}\) is an isomorphism, it follows that
\begin{equation}\label{pi_To_Theta}
\pi(\mathfrak e_{K,n}^1\mathfrak z)
=
\Theta'_{K/k,S}(1-n).
\end{equation}
Now put
\[
\Delta_{\Sigma}:=
\prod_{v\in\Sigma}
\bigl(1-\sigma_v^{-1}\operatorname N v^n\bigr)
\in \mathbb Z_p[G].
\]
The distinguished triangle
\[
\mathcal C_{K,S,\Sigma}^{\bullet}
\to
\mathcal C_{K,S}^{\bullet}
\to
\bigoplus_{w\in\Sigma_K}R\Gamma(\kappa(w),\mathbb Z_p(n))[1]
\to
\]
and the description of the local term as the two-term complex
\[
\left[
\bigoplus_{v\in\Sigma}\mathbb Z_p[G]
\stackrel{(1-\sigma_v^{-1}\operatorname N v^n)_{v\in\Sigma}}{\longrightarrow}
\bigoplus_{v\in\Sigma}\mathbb Z_p[G]
\right]
\]
show, via the determinant isomorphism attached to the triangle, that the
element
\(
\mathfrak z_{\Sigma}:=\Delta_{\Sigma}\cdot\mathfrak z,
\)
is a generator of
\(
\operatorname{det}_{\mathbb Z_p[G]}(\mathcal C_{K,S,\Sigma}^{\bullet}).
\)
Here the notation \(\mathfrak z_{\Sigma}:=\Delta_{\Sigma}\cdot\mathfrak z\) is
understood through the above determinant isomorphism. Therefore
\[
\pi_{\Sigma}(\mathfrak e_{K,n}^1\mathfrak z_{\Sigma})
=
\Delta_{\Sigma}\,
\pi(\mathfrak e_{K,n}^1\mathfrak z),
\]
where \(\pi_{\Sigma}\) denotes the analogous passage-to-cohomology map for
\(\mathcal C_{K,S,\Sigma}^{\bullet}\). Consequently, by (\ref{pi_To_Theta}), we have
\begin{equation}\label{pi_To_DeltaTheta}
\pi_{\Sigma}(\mathfrak e_{K,n}^1\mathfrak z_{\Sigma})
=
\Delta_{\Sigma}
\Theta'_{K/k,S}(1-n).
\end{equation}
 Let \(b_1^*,\ldots,b_m^*\) be the basis of
\(\operatorname{Hom}_{\mathfrak e^1_{K,n}\mathbb Z_p[G]}(F,\mathfrak e^1_{K,n}\mathbb Z_p[G])\) dual to \(b_1,\ldots,b_m\), and set
\[
\Psi_i:=b_i^*\circ \Psi\in \operatorname{Hom}_{\mathfrak e^1_{K,n}\mathbb Z_p[G]}(F,\mathfrak e^1_{K,n}\mathbb Z_p[G]).
\]
Following the convention of \cite[\S 4.1]{BKS0}, we write
\[
\Psi_2\wedge\cdots\wedge\Psi_m
\]
for the rank-reduction homomorphism
\[
\bigwedge_{\mathfrak e^1_{K,n}\mathbb Z_p[G]}^m F
\longrightarrow F.
\]
Explicitly, on a decomposable element \(x_1\wedge\cdots\wedge x_m\), it is
given by
\[
(\Psi_2\wedge\cdots\wedge\Psi_m)(x_1\wedge\cdots\wedge x_m)
=
\sum_{i=1}^m
(-1)^{i+1}
\det\bigl(\Psi_j(x_\ell)\bigr)_{\substack{2\leq j\leq m\\ \ell\neq i}}
x_i.
\]
We use the same notation after extension of scalars to \(\mathbb C_p\).
\begin{proposition}
\label{prop:Sigma-integrality-Theta-prime}
Let $\Sigma$ be a non-empty finite set of finite places of $k$ disjoint from
$S$, and put
\(
\Delta_{\Sigma}:=
\prod_{v\in\Sigma}
\bigl(1-\sigma_v^{-1}\operatorname N v^n\bigr)
\).
Then
\[
\Delta_{\Sigma}\Theta'_{K/k,S}(1-n)
\in
\mathbb Z_pK_{2n-1}(O_K)_{\mathrm{tf}}.
\]
\end{proposition}

\begin{proof}
Let
\[
F\stackrel{\Psi}{\longrightarrow}F
\]
be the representative of
\(\widetilde{\mathcal C}_{K,S,\Sigma}^{\bullet}\) constructed in Lemma
\ref{lem:quadratic-representative-Sigma-cutout}, and let
\(b_1,\ldots,b_m\) be the chosen basis of the degree-one term. After tensoring
with \(\mathbb C_p\), the finite module
\(
\mathfrak e^1_{K,n}
H^2_{\Sigma}(\operatorname{Spec}(O_{K,S}),\mathbb Z_p(n))
\)
vanishes. Hence, by construction,
\(
\operatorname{coker}(\mathbb C_p\Psi)
\)
is generated by the image of \(b_1\), and
\[
\operatorname{im}(\mathbb C_p\Psi)
=
\langle b_2,\ldots,b_m\rangle_{\mathfrak e^1_{K,n}\mathbb C_p[G]}.
\]
Therefore, by the rank-reduction formula of \cite[Lemma 4.3]{BKS0}, there exists
\[
a\in
\bigwedge_{\mathfrak e^1_{K,n}\mathbb Z_p[G]}^m F
\]
such that
\[
\pi_{\Sigma}(\mathfrak e^1_{K,n}\mathfrak z_{\Sigma})
=
(-1)^{m-1}
(\Psi_2\wedge\cdots\wedge\Psi_m)(a).
\]
In particular,
\[
\pi_{\Sigma}(\mathfrak e^1_{K,n}\mathfrak z_{\Sigma})\in F.
\]
On the other hand, by (\ref{pi_To_DeltaTheta}), we have
\[
\pi_{\Sigma}(\mathfrak e^1_{K,n}\mathfrak z_{\Sigma})
=
\Delta_{\Sigma}\Theta'_{K/k,S}(1-n).
\]
The right-hand side belongs, by definition, to
\[
\mathfrak e^1_{K,n}\mathbb C_pK_{2n-1}(O_K)
=
\mathbb C_p
H^0\bigl(\widetilde{\mathcal C}_{K,S,\Sigma}^{\bullet}\bigr).
\]
Thus it is killed by \(\mathbb C_p\Psi\). Since it also belongs to the integral
lattice \(F\), and since \(F\) is \(\mathbb Z_p\)-torsion-free, it follows that
\[
\Delta_{\Sigma}\Theta'_{K/k,S}(1-n)
\in
\ker(\Psi)
=
H^0\bigl(\widetilde{\mathcal C}_{K,S,\Sigma}^{\bullet}\bigr).
\]
Therefore
\[
\Delta_{\Sigma}\Theta'_{K/k,S}(1-n)
\in
\mathfrak e^1_{K,n}
H^1_{\Sigma}(\operatorname{Spec}(O_{K,S}),\mathbb Z_p(n)).
\]
By Lemma \ref{lem:Sigma-cohomology-torsion-free}, this module is
\(\mathbb Z_p\)-torsion-free and embeds into
\(
\mathbb Z_pK_{2n-1}(O_K)_{\mathrm{tf}}.
\)
This proves the claim.\qed
\end{proof}
We can now finish the proof of Theorem
\ref{Thm:annihilated-integrality-of-Theta-prime}.
Recall that \(S\) contains the archimedean places, the ramified places and the places above \(p\).
By \cite[Lemma 6.9(2)]{GreitherPopescu}, the ideal
\[
\operatorname{Ann}_{\mathbb Z_p[G]}
\bigl(\mathbb Z_pK_{2n-1}(O_K)_{\mathrm{tors}}\bigr)
\]
is generated over \(\mathbb Z_p[G]\) by the elements
\(
\Delta_{\Sigma}
=
\prod_{v\in\Sigma}
\bigl(1-\sigma_v^{-1}\operatorname N v^n\bigr),
\)
where \(\Sigma\) runs over the non-empty finite sets of finite places of \(k\)
which are disjoint from \(S\).
By Proposition \ref{prop:Sigma-integrality-Theta-prime}, for every such
\(\Sigma\), one has
\[
\Delta_{\Sigma}\Theta'_{K/k,S}(1-n)
\in
\mathbb Z_pK_{2n-1}(O_K)_{\mathrm{tf}}.
\]
It follows that
\[
\operatorname{Ann}_{\mathbb Z_p[G]}
\bigl(\mathbb Z_pK_{2n-1}(O_K)_{\mathrm{tors}}\bigr)
\cdot
\Theta'_{K/k,S}(1-n)
\subseteq
\mathbb Z_pK_{2n-1}(O_K)_{\mathrm{tf}}.
\]
This proves Theorem
\ref{Thm:annihilated-integrality-of-Theta-prime}.\qed
\subsection{A Coates--Sinnott Theorem for First Derivatives of \(L\)-Functions}\label{Coates_Sinnott_Section}
We shall use the following elementary observation. For every non-empty finite
set \(\Sigma\) of finite places of \(k\) disjoint from \(S\) and integer \(n\geq 2\), one has
\[
\Delta_\Sigma
=
\prod_{v\in\Sigma}
\bigl(1-\sigma_v^{-1}\operatorname Nv^n\bigr)
\in
\mathbb Q_p[G]^\times.
\]
Indeed, after extending scalars to \(\mathbb C_p\), the \(\chi\)-component of
\(\Delta_\Sigma\) is
\[
\chi(\Delta_\Sigma)
=
\prod_{v\in\Sigma}
\bigl(1-\chi(\sigma_v)^{-1}\operatorname Nv^n\bigr).
\]
Each factor is non-zero, since \(\chi(\sigma_v)^{-1}\) is a root of unity while
\(\operatorname Nv^n>1\). Hence \(\Delta_\Sigma\) is invertible in
\(\mathbb C_p[G]\). Since \(\Delta_\Sigma\in\mathbb Q_p[G]\) and becomes invertible after extension
of scalars to \(\mathbb C_p[G]\), it is already invertible in
\(\mathbb Q_p[G]\). Indeed, multiplication by \(\Delta_\Sigma\) on
\(\mathbb Q_p[G]\) becomes an isomorphism after tensoring with \(\mathbb C_p\);
hence its determinant is non-zero over \(\mathbb Q_p\).

We next define a fractional \(\mathbb Z_p[G]\)-ideal in \(\mathbb Q_p[G]\)
which will play a role in the annihilation of the even \(K\)-groups. Since \(\mathbb Z_pK_{2n-1}(O_K)_{\mathrm{tf}}\) is
\(\mathbb Z_p\)-torsion-free, every \(\mathbb Z_p[G]\)-homomorphism
\[
f\in
\operatorname{Hom}_{\mathbb Z_p[G]}
\bigl(
\mathbb Z_pK_{2n-1}(O_K)_{\mathrm{tf}},
\mathbb Z_p[G]
\bigr)
\]
extends uniquely to a \(\mathbb Q_p[G]\)-homomorphism
\[
\widetilde f\colon
\mathbb Q_pK_{2n-1}(O_K)_{\mathrm{tf}}
\longrightarrow
\mathbb Q_p[G].
\]
Indeed, for \(x\in \mathbb Z_pK_{2n-1}(O_K)_{\mathrm{tf}}\) and
\(a\in\mathbb Z_p\setminus\{0\}\), one necessarily has
\[
\widetilde f(a^{-1}x)=a^{-1}f(x).
\]
Thus we have a natural injective map
\[
\operatorname{Hom}_{\mathbb Z_p[G]}
\bigl(
\mathbb Z_pK_{2n-1}(O_K)_{\mathrm{tf}},
\mathbb Z_p[G]
\bigr)
\hookrightarrow
\operatorname{Hom}_{\mathbb Q_p[G]}
\bigl(
\mathbb Q_pK_{2n-1}(O_K)_{\mathrm{tf}},
\mathbb Q_p[G]
\bigr),
\qquad
f\longmapsto \widetilde f.
\]
By Proposition \ref{prop:Sigma-integrality-Theta-prime}, for every non-empty
\(\Sigma\) as above one has
\[
\Delta_\Sigma\Theta'_{K/k,S}(1-n)
\in
\mathbb Z_pK_{2n-1}(O_K)_{\mathrm{tf}}.
\]
Since \(\Delta_\Sigma\in\mathbb Q_p[G]^\times\), it follows that
\[
\Theta'_{K/k,S}(1-n)
\in
\mathbb Q_pK_{2n-1}(O_K)_{\mathrm{tf}}.
\]
We define \(\mathcal J(\Theta'_{K/k,S}(1-n))\) to be the fractional
\(\mathbb Z_p[G]\)-ideal of \(\mathbb Q_p[G]\) generated by the elements
\[
\widetilde f(\Theta'_{K/k,S}(1-n)),
\qquad
f\in
\operatorname{Hom}_{\mathbb Z_p[G]}
\bigl(
\mathbb Z_pK_{2n-1}(O_K)_{\mathrm{tf}},
\mathbb Z_p[G]
\bigr).
\]
By Theorem \ref{Thm:annihilated-integrality-of-Theta-prime}, one has
\[
\operatorname{Ann}_{\mathbb Z_p[G]}
\bigl(\mathbb Z_pK_{2n-1}(O_K)_{\mathrm{tors}}\bigr)
\cdot
\Theta'_{K/k,S}(1-n)
\subseteq
\mathbb Z_pK_{2n-1}(O_K)_{\mathrm{tf}}.
\]
It follows that
\[
\operatorname{Ann}_{\mathbb Z_p[G]}
\bigl(\mathbb Z_pK_{2n-1}(O_K)_{\mathrm{tors}}\bigr)
\cdot
\mathcal J\bigl(\Theta'_{K/k,S}(1-n)\bigr)
\subseteq
\mathbb Z_p[G].
\]
Indeed, if \(a\) belongs to the above annihilator and
\[
f\in
\operatorname{Hom}_{\mathbb Z_p[G]}
\bigl(
\mathbb Z_pK_{2n-1}(O_K)_{\mathrm{tf}},
\mathbb Z_p[G]
\bigr),
\]
then
\[
a\,\widetilde f(\Theta'_{K/k,S}(1-n))
=
\widetilde f\bigl(a\,\Theta'_{K/k,S}(1-n)\bigr)
=
f\bigl(a\,\Theta'_{K/k,S}(1-n)\bigr)
\in
\mathbb Z_p[G].
\]
This subsection is devoted to proving the following Coates--Sinnott type
annihilation result.
\begin{theorem}
\label{Main_Theo_02}
We have
\[
\operatorname{Ann}_{\mathbb Z_p[G]}
\bigl(\mathbb Z_pK_{2n-1}(O_K)_{\mathrm{tors}}\bigr)
\cdot
\mathcal J\bigl(\Theta'_{K/k,S}(1-n)\bigr)
\subseteq
\operatorname{Ann}_{\mathbb Z_p[G]}
\bigl(\mathbb Z_pK_{2n-2}(O_{K,S})\bigr).
\]
\end{theorem}

If \(\mathfrak e^1_{K,n}=0\), the theorem is immediate. We therefore assume
\(\mathfrak e^1_{K,n}\neq 0\).
The remainder of this subsection is devoted to the proof of Theorem \ref{Main_Theo_02}. 

\noindent We shall use the non-modified complex in this subsection. More precisely, we
cut out the \(\mathfrak e^1_{K,n}\)-part of \(\mathcal C_{K,S}^{\bullet}\) and
set
\[
\widetilde{\mathcal C}_{K,S}^{\bullet}
:=
\mathcal C_{K,S}^{\bullet}
\otimes_{\mathbb Z_p[G]}^{\mathrm L}
\mathfrak e^1_{K,n}\mathbb Z_p[G].
\]
Equivalently, this is the construction used above for
\(\mathcal C_{K,S,\Sigma}^{\bullet}\), in the special case \(\Sigma=\emptyset\).
Thus, by Proposition \ref{PROP-2-4}, the complex
\(\widetilde{\mathcal C}_{K,S}^{\bullet}\) is perfect over
\(\mathfrak e^1_{K,n}\mathbb Z_p[G]\), acyclic outside degrees \(0\) and \(1\),
and
\[
H^0\bigl(\widetilde{\mathcal C}_{K,S}^{\bullet}\bigr)
=
\mathfrak e^1_{K,n}\mathbb Z_pK_{2n-1}(O_K).
\]
Moreover, there is an exact sequence
\[
0
\longrightarrow
\mathfrak e^1_{K,n}\mathbb Z_pK_{2n-2}(O_{K,S})
\longrightarrow
H^1\bigl(\widetilde{\mathcal C}_{K,S}^{\bullet}\bigr)
\longrightarrow
\mathfrak e^1_{K,n}
\bigl(\mathbb{Z}_pY_{K,n}^{+}\bigr)
\longrightarrow
0.
\]
By Lemma \ref{lem:Y-free-over-ed}, one has an isomorphism
\[
\mathfrak e^1_{K,n}
\bigl(\mathbb{Z}_pY_{K,n}^{+}\bigr)
\simeq
\mathfrak e^1_{K,n}\mathbb Z_p[G].
\]
Since \(\mathfrak e^1_{K,n}\mathbb Z_p[G]\) is projective over itself, the
preceding exact sequence splits. We fix, once and for all, such a splitting and
therefore identify
\[
H^1\bigl(\widetilde{\mathcal C}_{K,S}^{\bullet}\bigr)
\simeq
\mathfrak e^1_{K,n}\mathbb Z_pK_{2n-2}(O_{K,S})
\oplus
\mathfrak e^1_{K,n}\mathbb Z_p[G].
\]
Let now \(\Sigma\) be a non-empty finite set of finite places of \(k\)
disjoint from \(S\), and put
\(
\Delta_{\Sigma}:=
\prod_{v\in\Sigma}
\bigl(1-\sigma_v^{-1}\operatorname N v^n\bigr).
\)
By Proposition \ref{prop:Sigma-integrality-Theta-prime}, the element
\[
\varepsilon_{\Sigma}:=
\Delta_{\Sigma}\Theta'_{K/k,S}(1-n)
\]
belongs to
\(
\mathbb Z_pK_{2n-1}(O_K)_{\mathrm{tf}}.
\)
Moreover, by construction,
\[
\widetilde r^B_{K,n}(\varepsilon_{\Sigma})
=
\Delta_{\Sigma}\theta'_{K/k,S}(1-n)
=
\theta'_{K/k,S,\Sigma}(1-n).
\]
where
\[
\theta'_{K/k,S,\Sigma}(1-n)
:=
\mathfrak e^1_{K,n}\theta^*_{K/k,S,\Sigma}(1-n).
\]
Since 
\(\theta'_{K/k,S,\Sigma}(1-n)\) is invertible in
\(\mathfrak e^1_{K,n}\mathbb C_p[G]\), the element
\(\varepsilon_{\Sigma}\) has trivial annihilator over
\(\mathfrak e^1_{K,n}\mathbb Z_p[G]\). Hence
\(
\varepsilon_{\Sigma}\mathbb Z_p[G]
\)
is a free \(\mathfrak e^1_{K,n}\mathbb Z_p[G]\)-module of rank one.
We now define a perfect complex of \(\mathfrak e^1_{K,n}\mathbb Z_p[G]\)-modules
\[
\mathcal D_{\Sigma}^{\bullet}:=
\left[
\varepsilon_{\Sigma}\mathbb Z_p[G]
\stackrel{0}{\longrightarrow}
\mathfrak e^1_{K,n}\mathbb Z_p[G]
\right],
\]
where the two terms are placed in degrees \(0\) and \(1\), respectively. 
Next, we define a morphism
\[
\mathcal D_{\Sigma}^{\bullet}
\longrightarrow
\widetilde{\mathcal C}_{K,S}^{\bullet}
\]
in \(D(\mathfrak e^1_{K,n}\mathbb Z_p[G])\). Choose a lift
\(
\widetilde\varepsilon_{\Sigma}
\in
\mathfrak e^1_{K,n}\mathbb Z_pK_{2n-1}(O_K)
=
H^0\bigl(\widetilde{\mathcal C}_{K,S}^{\bullet}\bigr)
\)
of \(\varepsilon_{\Sigma}\) under the natural projection
\[
\mathfrak e^1_{K,n}\mathbb Z_pK_{2n-1}(O_K)
\longrightarrow
\mathfrak e^1_{K,n}\mathbb Z_pK_{2n-1}(O_K)_{\mathrm{tf}}.
\]
We define the map on degree-zero cohomology by
\[
\varepsilon_{\Sigma}\mathbb Z_p[G]
\longrightarrow
H^0\bigl(\widetilde{\mathcal C}_{K,S}^{\bullet}\bigr),
\qquad
x\varepsilon_{\Sigma}
\longmapsto
x\widetilde\varepsilon_{\Sigma}.
\]
On degree-one cohomology, using the fixed splitting
\[
H^1\bigl(\widetilde{\mathcal C}_{K,S}^{\bullet}\bigr)
\simeq
\mathfrak e^1_{K,n}\mathbb Z_pK_{2n-2}(O_{K,S})
\oplus
\mathfrak e^1_{K,n}\mathbb Z_p[G],
\]
we take the natural inclusion
\[
\mathfrak e^1_{K,n}\mathbb Z_p[G]
\hookrightarrow
H^1\bigl(\widetilde{\mathcal C}_{K,S}^{\bullet}\bigr).
\]
Since both
\[
H^0(\mathcal D_{\Sigma}^{\bullet})
=
\varepsilon_{\Sigma}\mathbb Z_p[G]
\quad\text{and}\quad
H^1(\mathcal D_{\Sigma}^{\bullet})
=
\mathfrak e^1_{K,n}\mathbb Z_p[G]
\]
are projective \(\mathfrak e^1_{K,n}\mathbb Z_p[G]\)-modules, the higher
\(\operatorname{Ext}\)-terms in the hyper-Ext spectral sequence computing
\[
\operatorname{Hom}_{D(\mathfrak e^1_{K,n}\mathbb Z_p[G])}
\left(
\mathcal D_{\Sigma}^{\bullet},
\widetilde{\mathcal C}_{K,S}^{\bullet}
\right)
\]
vanish. Therefore the two cohomology maps above determine a unique morphism in
the derived category, which we denote by
\[
\iota_{\Sigma}\colon
\mathcal D_{\Sigma}^{\bullet}
\longrightarrow
\widetilde{\mathcal C}_{K,S}^{\bullet}.
\]
\begin{proposition}
\label{prop:cone-iota-Sigma}
The complex
\(
\operatorname{Cone}(\iota_{\Sigma})
\)
is a perfect complex of
\(\mathfrak e^1_{K,n}\mathbb Z_p[G]\)-modules, acyclic outside degrees
\(0\) and \(1\), with finite cohomology groups
\[
H^0(\operatorname{Cone}(\iota_{\Sigma}))
\simeq
{
\mathfrak e^1_{K,n}\mathbb Z_pK_{2n-1}(O_K)
}/{
\widetilde\varepsilon_{\Sigma}\mathbb Z_p[G]
}
\qquad
and
\qquad
H^1(\operatorname{Cone}(\iota_{\Sigma}))
\simeq
\mathfrak e^1_{K,n}\mathbb Z_pK_{2n-2}(O_{K,S}).
\]
Moreover, under the rational trivialization induced by the distinguished
triangle
\[
\mathcal D_{\Sigma}^{\bullet}
\longrightarrow
\widetilde{\mathcal C}_{K,S}^{\bullet}
\longrightarrow
\operatorname{Cone}(\iota_{\Sigma})
\longrightarrow
\mathcal D_{\Sigma}^{\bullet}[1],
\]
one has
\[
\operatorname{det}^{-1}_{\mathfrak e^1_{K,n}\mathbb Z_p[G]}
\bigl(\operatorname{Cone}(\iota_{\Sigma})\bigr)
=
\Delta_{\Sigma}
\mathfrak e^1_{K,n}\mathbb Z_p[G].
\]
\end{proposition}

\begin{proof}
The perfectness follows immediately from the fact that both
\(\mathcal D_{\Sigma}^{\bullet}\) and
\(\widetilde{\mathcal C}_{K,S}^{\bullet}\) are perfect. The long exact
sequence of cohomology gives
\[
0\to
\varepsilon_{\Sigma}\mathbb Z_p[G]
\to
\mathfrak e^1_{K,n}\mathbb Z_pK_{2n-1}(O_K)
\to
H^0(\operatorname{Cone}(\iota_{\Sigma}))
\to
\mathfrak e^1_{K,n}\mathbb Z_p[G]
\to
H^1(\widetilde{\mathcal C}_{K,S}^{\bullet})
\to
H^1(\operatorname{Cone}(\iota_{\Sigma}))
\to 0.
\]
By construction, the first map sends
\(x\varepsilon_{\Sigma}\) to \(x\widetilde\varepsilon_{\Sigma}\). It is
injective because passing to the torsion-free quotient identifies it with the
injective map
\[
\varepsilon_{\Sigma}\mathbb Z_p[G]
\hookrightarrow
\mathbb Z_pK_{2n-1}(O_K)_{\mathrm{tf}}.
\]
The map $\mathfrak e^1_{K,n}\mathbb Z_p[G]
\to
H^1(\widetilde{\mathcal C}_{K,S}^{\bullet})$ is the inclusion of
\(\mathfrak e^1_{K,n}\mathbb Z_p[G]\) into the fixed free direct summand of
\[
H^1(\widetilde{\mathcal C}_{K,S}^{\bullet})
\simeq
\mathfrak e^1_{K,n}\mathbb Z_pK_{2n-2}(O_{K,S})
\oplus
\mathfrak e^1_{K,n}\mathbb Z_p[G].
\]
The asserted descriptions of \(H^0\) and \(H^1\) follow. These groups are
finite because
\(\widetilde\varepsilon_{\Sigma}\mathbb Z_p[G]\) is a full lattice in
\(\mathfrak e^1_{K,n}\mathbb Q_pK_{2n-1}(O_K)\), and
\(\mathbb Z_pK_{2n-2}(O_{K,S})\) is finite.
It remains to record the determinant. The determinant isomorphism attached to
the above distinguished triangle gives
\[
\operatorname{det}(\operatorname{Cone}(\iota_{\Sigma}))
=
\operatorname{det}(\widetilde{\mathcal C}_{K,S}^{\bullet})
\cdot
\operatorname{det}(\mathcal D_{\Sigma}^{\bullet})^{-1}.
\]
Under the passage to cohomology fixed above and the ETNC, the determinant of
\(\widetilde{\mathcal C}_{K,S}^{\bullet}\) is generated by
\(\Theta'_{K/k,S}(1-n)\), whereas the determinant of
\(\mathcal D_{\Sigma}^{\bullet}\) is generated by
\[
\varepsilon_{\Sigma}
=
\Delta_{\Sigma}\Theta'_{K/k,S}(1-n).
\]
Hence
\(
\operatorname{det}^{-1}(\operatorname{Cone}(\iota_{\Sigma}))
=
\Delta_{\Sigma}
\mathfrak e^1_{K,n}\mathbb Z_p[G],
\)
as claimed.\qed
\end{proof}
We shall use the following elementary dévissage for perfect complexes with
finite cohomology. The proof given below follows an argument
kindly communicated to the author by C. Greither in the case \(\Lambda=\mathbb Z_p[G]\). We record it in the
slightly more general form needed here.
\begin{lemma}
\label{lem:Greither-perfect-finite-order}
Let \(\Lambda\) be a commutative \(\mathbb Z_p\)-order in a finite-dimensional
\(\mathbb Q_p\)-algebra. Let \(C^{\bullet}\) be a perfect complex of
\(\Lambda\)-modules, acyclic outside degrees \(0\) and \(1\), and assume that
\(H^0(C^{\bullet})\) and \(H^1(C^{\bullet})\) are finite. Then there exists an
exact sequence of \(\Lambda\)-modules
\[
0
\longrightarrow
H^0(C^{\bullet})
\longrightarrow
Q^0
\stackrel{d}{\longrightarrow}
Q^1
\longrightarrow
H^1(C^{\bullet})
\longrightarrow
0
\]
such that \(Q^0\) and \(Q^1\) are finite \(\Lambda\)-modules of projective
dimension at most \(1\). Moreover, \(C^{\bullet}\) is isomorphic in the derived category
\(D(\Lambda)\) of \(\Lambda\)-modules to the complex
\(
[Q^0
\stackrel{d}{\longrightarrow}
Q^1],
\)
where \(Q^0\) and \(Q^1\) are placed in degrees \(0\) and \(1\), respectively.
\end{lemma}

\begin{proof}
Since \(C^{\bullet}\) is perfect and acyclic outside degrees \(0\) and \(1\),
the usual truncation argument gives (e.g., \cite[\S 3]{Burns00}) a representative
\[
C^0
\stackrel{d}{\longrightarrow}
P
\]
of \(C^{\bullet}\), where \(P\) is finitely generated projective and \(C^0\)
is finitely generated of finite projective dimension over \(\Lambda\). Thus one
has an exact sequence
\[
0
\longrightarrow
H^0(C^{\bullet})
\longrightarrow
C^0
\stackrel{d}{\longrightarrow}
P
\longrightarrow
H^1(C^{\bullet})
\longrightarrow
0.
\]
Choose an integer \(a>0\) such that \(p^a\) annihilates
\(H^1(C^{\bullet})\) and the \(\mathbb Z_p\)-torsion submodule of \(C^0\).
Then \(p^aP\subseteq d(C^0)\). Let \(A\) be the inverse image of \(p^aP\) in
\(C^0\). The map \(d\) induces a surjection
\(
A\longrightarrow p^aP.
\)
Multiplication by \(p^a\) gives a map
\(
p^aA\longrightarrow p^{2a}P.
\)
This map is an isomorphism. It is surjective by construction. To prove
injectivity, let \(y=p^ax\in p^aA\) be in the kernel. Then
\(d(y)=p^ad(x)=0\). Since \(P\) is \(\mathbb Z_p\)-torsion-free, this implies
\(d(x)=0\), hence \(x\in H^0(C^{\bullet})\). By the choice of \(a\), the group
\(H^0(C^{\bullet})\) is killed by \(p^a\), and so \(y=p^ax=0\).
Set
\[
Q^0:=C^0/p^aA,
\qquad
Q^1:=P/p^{2a}P.
\]
Then \(Q^0\) and \(Q^1\) are finite, and the quotient complex
\(
[Q^0
\longrightarrow
Q^1]
\)
fits into an exact sequence
\[
0
\longrightarrow
H^0(C^{\bullet})
\longrightarrow
Q^0
\longrightarrow
Q^1
\longrightarrow
H^1(C^{\bullet})
\longrightarrow
0.
\]
Moreover, the kernel complex
\[
p^aA
\longrightarrow
p^{2a}P
\]
is acyclic, since the displayed map is an isomorphism. Hence
\(
C^{\bullet}\simeq [Q^0\to Q^1]
\)
in \(D(\Lambda)\).
It remains to check the projective dimensions. The module \(Q^1\) has
projective dimension at most \(1\), since it is the cokernel of the injective
map
\[
P
\stackrel{p^{2a}}{\longrightarrow}
P.
\]
The module \(Q^0\) is finite and has finite projective dimension, since it is
the cokernel of the injection
\[
p^aA\hookrightarrow C^0,
\]
with \(p^aA\simeq p^{2a}P\) projective and \(C^0\) of finite projective
dimension.
Finally, let \(\mathfrak m\) be a maximal ideal of \(\Lambda\). Since
\(\Lambda\) is a \(\mathbb Z_p\)-order, \(\Lambda_{\mathfrak m}\) is a
one-dimensional local ring of depth \(1\). A non-zero finite \(\Lambda_{\mathfrak m}\)-
module has depth \(0\). The Auslander--Buchsbaum formula therefore shows that
any finite \(\Lambda_{\mathfrak m}\)-module of finite projective dimension has
projective dimension at most \(1\). Applying this to \(Q^0_{\mathfrak m}\) for
all \(\mathfrak m\), we get
\[
\operatorname{pd}_{\Lambda}(Q^0)\leq 1.
\]
This proves the lemma.\qed
\end{proof}
\begin{theorem}
\label{thm:Fitting-relation-cone-Sigma}
For every non-empty finite set $\Sigma$ of finite places of $k$ disjoint from
$S$, one has
\[
\begin{aligned}
\operatorname{Fitt}_{\mathfrak e^1_{K,n}\mathbb Z_p[G]}
\left(
\left(
\frac{
\mathfrak e^1_{K,n}\mathbb Z_pK_{2n-1}(O_K)
}{
\widetilde\varepsilon_{\Sigma}\mathbb Z_p[G]
}
\right)^{\vee}
\right)^\#
&=
\Delta_{\Sigma}\,
\operatorname{Fitt}_{\mathfrak e^1_{K,n}\mathbb Z_p[G]}
\left(
\mathfrak e^1_{K,n}\mathbb Z_pK_{2n-2}(O_{K,S})
\right).
\end{aligned}
\]
\end{theorem}

\begin{proof}
Applying Lemma \ref{lem:Greither-perfect-finite-order} to
\(\operatorname{Cone}(\iota_{\Sigma})\), we obtain an exact sequence of
\(\mathfrak e^1_{K,n}\mathbb Z_p[G]\)-modules
\[
0
\longrightarrow
H^0(\operatorname{Cone}(\iota_{\Sigma}))
\longrightarrow
Q^0
\longrightarrow
Q^1
\longrightarrow
H^1(\operatorname{Cone}(\iota_{\Sigma}))
\longrightarrow
0,
\]
where \(Q^0\) and \(Q^1\) are finite
\(\mathfrak e^1_{K,n}\mathbb Z_p[G]\)-modules of projective dimension at most
one, and
\[
\operatorname{Cone}(\iota_{\Sigma})
\simeq
[Q^0\longrightarrow Q^1]
\]
in \(D(\mathfrak e^1_{K,n}\mathbb Z_p[G])\).
By \cite[Lemma 5]{BunsG0}, applied to the preceding exact sequence, one has
\[
\operatorname{Fitt}_{\mathfrak e^1_{K,n}\mathbb Z_p[G]}
\bigl(
H^0(\operatorname{Cone}(\iota_{\Sigma}))^\vee
\bigr)^\#
\operatorname{Fitt}_{\mathfrak e^1_{K,n}\mathbb Z_p[G]}(Q^1)
=
\operatorname{Fitt}_{\mathfrak e^1_{K,n}\mathbb Z_p[G]}
\bigl(
H^1(\operatorname{Cone}(\iota_{\Sigma}))
\bigr)
\operatorname{Fitt}_{\mathfrak e^1_{K,n}\mathbb Z_p[G]}(Q^0).
\]
The involution \(\#\) appears
because the convention for the dual action in \cite[Lemma 5]{BunsG0} is the
opposite one.
Since \(Q^0\) and \(Q^1\) are finite modules of projective dimension at most one, their
Fitting ideals are principal ideals generated by non-zero-divisors; hence they
may be inverted after extension to
\(\mathfrak e^1_{K,n}\mathbb Q_p[G]\). Moreover, from the representative
\([Q^0\to Q^1]\), one has
\[
\operatorname{det}^{-1}_{\mathfrak e^1_{K,n}\mathbb Z_p[G]}
\bigl(\operatorname{Cone}(\iota_{\Sigma})\bigr)
=
\operatorname{Fitt}_{\mathfrak e^1_{K,n}\mathbb Z_p[G]}(Q^0)
\operatorname{Fitt}_{\mathfrak e^1_{K,n}\mathbb Z_p[G]}(Q^1)^{-1}.
\]
By Proposition \ref{prop:cone-iota-Sigma},
\[
\operatorname{det}^{-1}_{\mathfrak e^1_{K,n}\mathbb Z_p[G]}
\bigl(\operatorname{Cone}(\iota_{\Sigma})\bigr)
=
\Delta_{\Sigma}\mathfrak e^1_{K,n}\mathbb Z_p[G].
\]
Combining this determinant computation with the preceding Fitting equality gives
\[
\operatorname{Fitt}_{\mathfrak e^1_{K,n}\mathbb Z_p[G]}
\bigl(H^0(\operatorname{Cone}(\iota_{\Sigma}))^{\vee}\bigr)^\#
=
\Delta_{\Sigma}\,
\operatorname{Fitt}_{\mathfrak e^1_{K,n}\mathbb Z_p[G]}
\bigl(H^1(\operatorname{Cone}(\iota_{\Sigma}))\bigr).
\]
The claim now follows from Proposition \ref{prop:cone-iota-Sigma}.\qed
\end{proof}
Before proving Theorem \ref{Main_Theo_02}, we record three standard facts.
\begin{itemize}
\item The order \(\mathbb Z_p[G]\) is Gorenstein
\cite[Cor. 10.29]{CurtisReiner}. In particular, it has injective dimension
one as a module over itself, and if \(A\) is a finitely generated
\(\mathbb Z_p\)-torsion-free \(\mathbb Z_p[G]\)-module, then
\begin{equation}
\label{Gorenstein-Ext-zero}
\operatorname{Ext}_{\mathbb Z_p[G]}^{1}(A,\mathbb Z_p[G])=0.
\end{equation}
See, for instance, \cite[\S A.3, (b), and (A.7)]{BS19}.

\item If \(B\) is a finite \(\mathbb Z_p[G]\)-module, then the exact sequence
\[
0\to
\mathbb Z_p[G]
\to
\mathbb Q_p[G]
\to
\mathbb Q_p[G]/\mathbb Z_p[G]
\to 0
\]
gives a canonical isomorphism
\[
\operatorname{Ext}_{\mathbb Z_p[G]}^{1}(B,\mathbb Z_p[G])
\simeq
\operatorname{Hom}_{\mathbb Z_p[G]}
\bigl(B,\mathbb Q_p[G]/\mathbb Z_p[G]\bigr).
\]
Moreover, under the usual identification
\[
\operatorname{Hom}_{\mathbb Z_p[G]}
\bigl(B,\mathbb Q_p[G]/\mathbb Z_p[G]\bigr)
\simeq
\bigl(B^\vee\bigr)^\#,
\]
we get
\begin{equation}
\label{Gorenstein-Ext-dual}
\operatorname{Ext}_{\mathbb Z_p[G]}^{1}(B,\mathbb Z_p[G])
\simeq
\bigl(B^\vee\bigr)^\#.
\end{equation}

\item Finally, if \(M\) is any finitely generated \(\mathbb Z_p[G]\)-module,
then
\begin{equation}
\label{Hom-torsion-free-quotient}
\operatorname{Hom}_{\mathbb Z_p[G]}(M,\mathbb Z_p[G])
=
\operatorname{Hom}_{\mathbb Z_p[G]}(M_{\mathrm{tf}},\mathbb Z_p[G]),
\end{equation}
because \(\mathbb Z_p[G]\) is \(\mathbb Z_p\)-torsion-free.
\end{itemize}
\begin{corollary}
Theorem \ref{Main_Theo_02} holds.
\end{corollary}

\begin{proof}
Fix a non-empty finite set \(\Sigma\) of finite places of \(k\) disjoint from
\(S\). We begin by applying
\(\operatorname{Hom}_{\mathbb Z_p[G]}(-,\mathbb Z_p[G])\) to the exact sequence
\[
0\to
\widetilde\varepsilon_{\Sigma}\mathbb Z_p[G]
\to
\mathfrak e^1_{K,n}\mathbb Z_pK_{2n-1}(O_K)
\to
\frac{
\mathfrak e^1_{K,n}\mathbb Z_pK_{2n-1}(O_K)
}{
\widetilde\varepsilon_{\Sigma}\mathbb Z_p[G]
}
\to 0.
\]
This gives an exact sequence
\begin{small}
\begin{align}
\label{Exact-Seq-dual-Epsilon}
&
\operatorname{Hom}_{\mathbb Z_p[G]}
\bigl(\mathfrak e^1_{K,n}\mathbb Z_pK_{2n-1}(O_K),\mathbb Z_p[G]\bigr)
\to
\operatorname{Hom}_{\mathbb Z_p[G]}
\bigl(\widetilde\varepsilon_{\Sigma}\mathbb Z_p[G],\mathbb Z_p[G]\bigr)
\to
\operatorname{Ext}^1_{\mathbb Z_p[G]}
\left(
\frac{
\mathfrak e^1_{K,n}\mathbb Z_pK_{2n-1}(O_K)
}{
\widetilde\varepsilon_{\Sigma}\mathbb Z_p[G]
},
\mathbb Z_p[G]
\right)
\nonumber\\
\to &
\operatorname{Ext}^1_{\mathbb Z_p[G]}
\bigl(\mathfrak e^1_{K,n}\mathbb Z_pK_{2n-1}(O_K),\mathbb Z_p[G]\bigr)
\to
\operatorname{Ext}^1_{\mathbb Z_p[G]}
\bigl(\widetilde\varepsilon_{\Sigma}\mathbb Z_p[G],\mathbb Z_p[G]\bigr).
\end{align}
\end{small}
Since \(\widetilde\varepsilon_{\Sigma}\mathbb Z_p[G]\) is
\(\mathbb Z_p\)-torsion-free, \eqref{Gorenstein-Ext-zero} gives
\[
\operatorname{Ext}^1_{\mathbb Z_p[G]}
\bigl(\widetilde\varepsilon_{\Sigma}\mathbb Z_p[G],\mathbb Z_p[G]\bigr)=0.
\]
We apply \(\operatorname{Hom}_{\mathbb Z_p[G]}(-,\mathbb Z_p[G])\)
to the exact sequence
\[
0\to
\mathfrak e^1_{K,n}\mathbb Z_pK_{2n-1}(O_K)_{\mathrm{tors}}
\to
\mathfrak e^1_{K,n}\mathbb Z_pK_{2n-1}(O_K)
\to
\mathfrak e^1_{K,n}\mathbb Z_pK_{2n-1}(O_K)_{\mathrm{tf}}
\to 0.
\]
Using \eqref{Gorenstein-Ext-zero}, \eqref{Gorenstein-Ext-dual}, and the
fact that \(\mathbb Z_p[G]\) has injective dimension one over itself, this gives
\[
\operatorname{Ext}^1_{\mathbb Z_p[G]}
\bigl(\mathfrak e^1_{K,n}\mathbb Z_pK_{2n-1}(O_K),\mathbb Z_p[G]\bigr)
\simeq
\left(
\bigl(\mathfrak e^1_{K,n}\mathbb Z_pK_{2n-1}(O_K)_{\mathrm{tors}}\bigr)^\vee
\right)^\#.
\]
Moreover, the natural projection
\[
\mathfrak e^1_{K,n}\mathbb Z_pK_{2n-1}(O_K)
\longrightarrow
\mathfrak e^1_{K,n}\mathbb Z_pK_{2n-1}(O_K)_{\mathrm{tf}}
\]
restricts to an isomorphism
\[
\widetilde\varepsilon_{\Sigma}\mathbb Z_p[G]
\xrightarrow{\ \simeq\ }
\varepsilon_{\Sigma}\mathbb Z_p[G].
\]
Indeed, if \(a\widetilde\varepsilon_{\Sigma}\) maps to zero, then
\(a\varepsilon_{\Sigma}=0\). Since
\(\varepsilon_{\Sigma}\mathbb Z_p[G]\) is a free
\(\mathfrak e^1_{K,n}\mathbb Z_p[G]\)-module of rank one, this forces
\(\mathfrak e^1_{K,n}a=0\), and hence
\(a\widetilde\varepsilon_{\Sigma}=0\).
Thus we shall make the identification
\[
\operatorname{Hom}_{\mathbb Z_p[G]}
\bigl(\widetilde\varepsilon_{\Sigma}\mathbb Z_p[G],\mathbb Z_p[G]\bigr)
\simeq
\operatorname{Hom}_{\mathbb Z_p[G]}
\bigl(\varepsilon_{\Sigma}\mathbb Z_p[G],\mathbb Z_p[G]\bigr).
\]
Combining these identifications with \eqref{Exact-Seq-dual-Epsilon}, and using
\eqref{Hom-torsion-free-quotient}, we obtain an exact sequence
\begin{small}
\begin{align}
\label{Exact-Seq-Kernel-J-epsilon}
&
\operatorname{Hom}_{\mathbb Z_p[G]}
\bigl(\mathfrak e^1_{K,n}\mathbb Z_pK_{2n-1}(O_K)_{\mathrm{tf}},\mathbb Z_p[G]\bigr)
\to
\operatorname{Hom}_{\mathbb Z_p[G]}
\bigl(\varepsilon_{\Sigma}\mathbb Z_p[G],\mathbb Z_p[G]\bigr)
\to
\left(
\left(
\frac{
\mathfrak e^1_{K,n}\mathbb Z_pK_{2n-1}(O_K)
}{
\widetilde\varepsilon_{\Sigma}\mathbb Z_p[G]
}
\right)^\vee
\right)^\# \nonumber \\
\to
&\left(
\bigl(\mathfrak e^1_{K,n}\mathbb Z_pK_{2n-1}(O_K)_{\mathrm{tors}}\bigr)^\vee
\right)^\#
\to 0.
\end{align}
\end{small}
Let \(\mathcal J(\varepsilon_\Sigma)\) be the ideal of
\(\mathfrak e^1_{K,n}\mathbb Z_p[G]\) generated by the elements
\[
f(\varepsilon_\Sigma),
\qquad
f\in
\operatorname{Hom}_{\mathbb Z_p[G]}
\bigl(\mathbb Z_pK_{2n-1}(O_K)_{\mathrm{tf}},\mathbb Z_p[G]\bigr).
\]
Since \(\mathfrak e^1_{K,n}\varepsilon_\Sigma=\varepsilon_\Sigma\), this ideal
is equivalently the image of the evaluation map
\[
\operatorname{Hom}_{\mathbb Z_p[G]}
\bigl(
\mathfrak e^1_{K,n}\mathbb Z_pK_{2n-1}(O_K)_{\mathrm{tf}},
\mathbb Z_p[G]
\bigr)
\longrightarrow
\mathfrak e^1_{K,n}\mathbb Z_p[G],
\qquad
f\longmapsto f(\varepsilon_\Sigma).
\]
Let
\[
\mathcal K:=
\ker\left(
\left(
\left(
\frac{
\mathfrak e^1_{K,n}\mathbb Z_pK_{2n-1}(O_K)
}{
\widetilde\varepsilon_{\Sigma}\mathbb Z_p[G]
}
\right)^\vee
\right)^\#
\to
\left(
\bigl(\mathfrak e^1_{K,n}\mathbb Z_pK_{2n-1}(O_K)_{\mathrm{tors}}\bigr)^\vee
\right)^\#
\right).
\]
Then \eqref{Exact-Seq-Kernel-J-epsilon} gives a commutative diagram with exact
rows
\begin{small}
\begin{equation*}
\xymatrix@=1.5pc{
&\mathrm{Hom}_{\mathbb{Z}_p[G]}(\mathfrak e^1_{K,n}\mathbb Z_pK_{2n-1}(O_K)_{\mathrm{tf}}, \mathbb{Z}_p[G])\ar[r] \ar@{->>}[d]_{\Phi}&\mathrm{Hom}_{\mathbb{Z}_p[G]}(\varepsilon_{\Sigma}\mathbb Z_p[G], \mathbb Z_p[G]) \ar[r] \ar[d]^{\cong}_{\Phi} &\mathcal{K}\ar[r] \ar[d] &0\\
0 \ar[r] &\mathcal{J}(\varepsilon_\Sigma) \ar[r] & \mathfrak e^1_{K,n}\mathbb Z_p[G]\ar[r] &\mathfrak e^1_{K,n}\mathbb Z_p[G]/\mathcal{J}(\varepsilon_\Sigma)\ar[r] &0
}
\end{equation*}
\end{small}
where $\Phi:\;f\mapsto f(\varepsilon_\Sigma)$. The isomorphism in the middle follows from the fact that $\varepsilon_\Sigma\mathbb Z_p[G]$ is a rank-one free $\mathfrak e^1_{K,n}\mathbb Z_p[G]$-module and the vertical map on the right-hand side is induced by commutativity of the diagram. Hence the snake
lemma gives
\[
\mathcal K\simeq
\mathfrak e^1_{K,n}\mathbb Z_p[G]/\mathcal J(\varepsilon_\Sigma).
\]
Applying the involution \(\#\) to the resulting exact sequence gives
\[
0\to
\bigl(
\mathfrak e^1_{K,n}\mathbb Z_p[G]/\mathcal J(\varepsilon_\Sigma)
\bigr)^\#
\to
\left(
\frac{
\mathfrak e^1_{K,n}\mathbb Z_pK_{2n-1}(O_K)
}{
\widetilde\varepsilon_{\Sigma}\mathbb Z_p[G]
}
\right)^\vee
\to
\bigl(\mathfrak e^1_{K,n}\mathbb Z_pK_{2n-1}(O_K)_{\mathrm{tors}}\bigr)^\vee
\to 0.
\]
Therefore the standard Fitting ideal inclusion for short exact sequences gives
\begin{small}
\begin{align}
\label{Fitt-inclusion-before-main-relation}
&
\operatorname{Fitt}_{\mathfrak e^1_{K,n}\mathbb Z_p[G]}
\bigl(
(\mathfrak e^1_{K,n}\mathbb Z_pK_{2n-1}(O_K)_{\mathrm{tors}})^\vee
\bigr)
\cdot
\mathcal J(\varepsilon_\Sigma)^\#
\subseteq
\operatorname{Fitt}_{\mathfrak e^1_{K,n}\mathbb Z_p[G]}
\left(
\left(
\frac{
\mathfrak e^1_{K,n}\mathbb Z_pK_{2n-1}(O_K)
}{
\widetilde\varepsilon_{\Sigma}\mathbb Z_p[G]
}
\right)^\vee
\right).
\end{align}
\end{small}
If \(e_\chi\) occurs in \(\mathfrak e^1_{K,n}\), then so does
\(e_{\chi^{-1}}\) by \eqref{eq:35}. Hence
\[
(\mathfrak e^1_{K,n})^\#=\mathfrak e^1_{K,n}.
\]
Consequently,
\[
\bigl(\mathfrak e^1_{K,n}\mathbb Z_pK_{2n-1}(O_K)_{\mathrm{tors}}\bigr)^\vee
\simeq
\mathfrak e^1_{K,n}
\bigl(\mathbb Z_pK_{2n-1}(O_K)_{\mathrm{tors}}\bigr)^\vee.
\]
It follows that
\begin{align*}
\operatorname{Fitt}_{\mathfrak e^1_{K,n}\mathbb Z_p[G]}
\bigl(
(\mathfrak e^1_{K,n}\mathbb Z_pK_{2n-1}(O_K)_{\mathrm{tors}})^\vee
\bigr)
& =
\mathfrak e^1_{K,n}
\operatorname{Fitt}_{\mathbb Z_p[G]}
\bigl(
(\mathbb Z_pK_{2n-1}(O_K)_{\mathrm{tors}})^\vee
\bigr)
\\
& =
\mathfrak e^1_{K,n}
\operatorname{Ann}_{\mathbb Z_p[G]}
\bigl(
(\mathbb Z_pK_{2n-1}(O_K)_{\mathrm{tors}})^\vee
\bigr)
\\
& =
\mathfrak e^1_{K,n}
\operatorname{Ann}_{\mathbb Z_p[G]}
\bigl(
\mathbb Z_pK_{2n-1}(O_K)_{\mathrm{tors}}
\bigr)^\#.
\end{align*}
Here the second equality follows from the fact that
\(\mathbb Z_pK_{2n-1}(O_K)_{\mathrm{tors}}\) is cyclic over
\(\mathbb Z_p[G]\). Indeed, by the Quillen--Lichtenbaum identification and the
boundary map associated with
\[
0\to \mathbb Z_p(n)\to \mathbb Q_p(n)\to
\mathbb Q_p/\mathbb Z_p(n)\to 0,
\]
this torsion subgroup identifies with
\[
H^0(K,\mathbb Q_p/\mathbb Z_p(n)),
\]
which is cyclic as a \(\mathbb Z_p\)-module, hence also cyclic as a
\(\mathbb Z_p[G]\)-module. Thus its Pontryagin dual is cyclic, and for a cyclic
module the Fitting ideal coincides with the annihilator. The last equality
follows, for example, from \cite[Lemma 3.1(ii)]{Popescu}. Since
\(\mathfrak e^1_{K,n}\mathcal J(\varepsilon_\Sigma)=\mathcal J(\varepsilon_\Sigma)\),
\eqref{Fitt-inclusion-before-main-relation} gives
\[
\operatorname{Ann}_{\mathbb Z_p[G]}
\bigl(
\mathbb Z_pK_{2n-1}(O_K)_{\mathrm{tors}}
\bigr)^\#
\cdot
\mathcal J(\varepsilon_\Sigma)^\#
\subseteq
\operatorname{Fitt}_{\mathfrak e^1_{K,n}\mathbb Z_p[G]}
\left(
\left(
\frac{
\mathfrak e^1_{K,n}\mathbb Z_pK_{2n-1}(O_K)
}{
\widetilde\varepsilon_{\Sigma}\mathbb Z_p[G]
}
\right)^\vee
\right).
\]
Applying \(\#\) to this inclusion and using
Theorem \ref{thm:Fitting-relation-cone-Sigma}, we obtain
\[
\operatorname{Ann}_{\mathbb Z_p[G]}
\bigl(
\mathbb Z_pK_{2n-1}(O_K)_{\mathrm{tors}}
\bigr)
\cdot
\mathcal J(\varepsilon_\Sigma)
\subseteq
\Delta_{\Sigma}
\operatorname{Fitt}_{\mathfrak e^1_{K,n}\mathbb Z_p[G]}
\left(
\mathfrak e^1_{K,n}\mathbb Z_pK_{2n-2}(O_{K,S})
\right).
\]
Finally, by definition of \(\varepsilon_\Sigma\), one has
\[
\varepsilon_\Sigma
=
\Delta_\Sigma\Theta'_{K/k,S}(1-n).
\]
Therefore
\[
\mathcal J(\varepsilon_\Sigma)
=
\Delta_\Sigma
\mathcal J\bigl(\Theta'_{K/k,S}(1-n)\bigr).
\]
Thus
\[
\Delta_\Sigma\,
\operatorname{Ann}_{\mathbb Z_p[G]}
\bigl(
\mathbb Z_pK_{2n-1}(O_K)_{\mathrm{tors}}
\bigr)
\cdot
\mathcal J\bigl(\Theta'_{K/k,S}(1-n)\bigr)
\subseteq
\Delta_{\Sigma}
\operatorname{Fitt}_{\mathfrak e^1_{K,n}\mathbb Z_p[G]}
\left(
\mathfrak e^1_{K,n}\mathbb Z_pK_{2n-2}(O_{K,S})
\right).
\]
Since \(\Delta_\Sigma\in\mathbb Q_p[G]^\times\), we may cancel
\(\Delta_\Sigma\) after extending scalars to \(\mathbb Q_p[G]\). Hence
\[
\operatorname{Ann}_{\mathbb Z_p[G]}
\bigl(
\mathbb Z_pK_{2n-1}(O_K)_{\mathrm{tors}}
\bigr)
\cdot
\mathcal J\bigl(\Theta'_{K/k,S}(1-n)\bigr)
\subseteq
\operatorname{Fitt}_{\mathfrak e^1_{K,n}\mathbb Z_p[G]}
\left(
\mathfrak e^1_{K,n}\mathbb Z_pK_{2n-2}(O_{K,S})
\right).
\]
The right-hand side is contained in
\[
\operatorname{Ann}_{\mathbb Z_p[G]}
\left(
\mathbb Z_pK_{2n-2}(O_{K,S})
\right)
\]
Indeed, if
\(a\in \mathfrak e^1_{K,n}\mathbb Z_p[G]\) annihilates
\(\mathfrak e^1_{K,n}\mathbb Z_pK_{2n-2}(O_{K,S})\), then \(a\) annihilates
the \((1-\mathfrak e^1_{K,n})\)-part automatically, and therefore annihilates
the whole module.\qed
\end{proof}

\end{document}